\documentclass[reqno,11pt,a4paper]{amsart}
\usepackage{amsmath}
\usepackage{amscd}
\usepackage{amsfonts}
\usepackage{amssymb}
\usepackage{amsthm}
\usepackage{latexsym}
\usepackage{graphicx}
\usepackage{subfigure}
\usepackage{mathrsfs}
\usepackage{enumerate}
\usepackage{hyperref}
\usepackage[headings]{fullpage}
\usepackage{color}
\usepackage{tikz}

\numberwithin{equation}{section}

\newtheorem{theorem}{Theorem}[section]
\newtheorem{lemma}[theorem]{Lemma}

\newtheorem{proposition}[theorem]{Proposition}

\theoremstyle{remark}
\newtheorem{remark}{Remark}[section]

\DeclareMathOperator{\Div}{div}
\renewcommand{\geq}{\geqslant}
\renewcommand{\ge}{\geqslant}
\renewcommand{\leq}{\leqslant}
\renewcommand{\le}{\leqslant}
\newcommand{\R}{\ensuremath{\mathbb{R}}}

\newcommand{\dd}{\mathrm{d}}
\newcommand{\Id}{\ensuremath{\mathrm{Id}}}
\renewcommand{\Re}{\operatorname{Re}}
\allowdisplaybreaks

\begin{document}
\title[Euler-Alignment System With Pressure]
{Global well-posedness and asymptotic behavior for the Euler-alignment system with pressure}

\author{Xiang Bai}
\address{Academy of Mathematics and Systems Science, Chinese Academy of Sciences, Beijing 100190, P.R. China}
\email{xiangbai@mail.bnu.edu.cn}
\author{Changhui Tan}
\address{Department of Mathematics, University of South Carolina, Columbia SC 29208, USA}
\email{tan@math.sc.edu}
\author{Liutang Xue}
\address{Laboratory of Mathematics and Complex Systems (MOE), School of Mathematical Sciences, Beijing Normal University, Beijing 100875, P.R. China}
\email{xuelt@bnu.edu.cn}
\thanks{{\it Acknowledgments.} C Tan was partially supported by NSF Grants DMS-2108264 and DMS-2238219. L Xue was partially supported by National Key Research and Development Program of China (No. 2020YFA0712900) and NSFC (No. 12271045).}

\begin{abstract}
We study the Cauchy problem of the compressible Euler system with strongly singular velocity alignment. We establish a global well-posedness theory for the system with small smooth initial data. Additionally, we derive asymptotic emergent behaviors for the system, providing time decay estimates with optimal decay rates. Notably, the optimal decay rate we obtain does not align with the corresponding fractional heat equation within our considered range, where the parameter $\alpha\in(0,1)$. This highlights the distinct feature of the alignment operator.

\end{abstract}

\maketitle \centerline{\date}
\section{Introduction}
We consider the Cauchy problem of the following Euler-alignment system in $\R_+\times\R^N$,
\begin{equation}\label{eq.EAS}
\begin{cases}
  \partial_t\rho+\Div(\rho u)=0,\\
  \partial_t(\rho u)+\Div(\rho u\otimes u) + \nabla P(\rho)=\mathcal{D}(u,\rho),\\
  (\rho,u)|_{t=0}=(\rho_0,u_0).
\end{cases}
\end{equation}
Here, $\rho$ is the density, and $u=(u_1,\cdots, u_N)$ is the velocity field. $P(\rho)$ stands for the pressure, given by the power law
\begin{equation}\label{eq:pressure}
  P(\rho)=\rho^{\gamma},\quad \gamma\ge 1.
\end{equation}
The pressure is known as isothermal when $\gamma=1$, and isentropic when $\gamma>1$.
The term $\mathcal{D}(u,\rho)$ represents the nonlocal velocity alignment.
It takes the form:
\begin{align}\label{eq:Du-rho}
  \mathcal{D}(u,\rho)(t,x)=-\rho(t,x)\int_{\R^N}\phi(x-y)\big(u(t,x)-u(t,y)\big)\rho(t,y)\dd y,
\end{align}
where $\phi$ is called the \emph{communication protocol}, which models the strength of the pairwise alignment interaction. 

The Euler-alignment system, represented by \eqref{eq.EAS}, serves as a model for capturing the collective behaviors exhibited by animal swarms. Over the past decade, there has been a growing interest in the literature regarding the analysis of the Euler-alignment system, including global wellposedness and the large time flocking behavior. See e.g. \cite{tadmor2014critical,carrillo2016critical,he2017global,choi2019global,tan2020euler,shvydkoy2021dynamics,lear2022existence,leslie2023sticky,black2024asymptotic}.

We are particularly interested in a specific family of communication protocols known as \emph{strongly singular communication}, wherein the function $\phi$ exhibits a non-integrable singularity at the origin. A prototypical instance of such protocols is given by:
\begin{equation}\label{eq:phialpha}
  \phi(x) = \phi_\alpha (x) =\mu\,\frac{ c_{\alpha,N} }{|x|^{N+\alpha}}, \quad\text{where}\quad\mu>0,\quad
  c_{\alpha,N} = \tfrac{2^{\alpha} \Gamma(\frac{N+\alpha}{2})}{\pi^{N/2}\Gamma(-\frac\alpha2)},\quad \alpha\in(0,2).
\end{equation}
It is evident that the velocity alignment term $\mathcal{D}(u,\rho)$ can be expressed in a commutator form:
\begin{equation}\label{eq:alignment}
  \mathcal{D}(u,\rho)= -\mu\rho \big(\Lambda^\alpha (\rho u)-u\Lambda^\alpha \rho\big).
\end{equation}
The presence of the singularity introduces dissipative characteristics to the system. Specifically, when we enforce $\rho\equiv1$, the alignment term transforms into the fractional Laplacian:
\[\mathcal{D}(u,1)=-\mu\Lambda^\alpha u:=\mu\,c_{\alpha,N}\int_{\R^N}\frac{u(x)-u(y)}{|x-y|^{N+\alpha}}\dd y,\]
resulting in a regularization effect on the solution. The corresponding system
\begin{equation}\label{eq:fractalBurgers}
	\partial_tu+u\cdot\nabla u=-\mu\Lambda^\alpha u
\end{equation}
is recognized as the fractal Burgers equation. This equation has been extensively studied in \cite{kiselev2008blow}, where global well-posedness is established for $\alpha\in[1,2)$, while solutions may exhibit shock formations for $\alpha\in(0,1)$.

The regularization effect of \eqref{eq:alignment} was explored in \cite{do2018global,shvydkoy2017eulerian,shvydkoy2018eulerian} for a one-dimensional periodic domain $\mathbb{T}$, under the assumption of no pressure (i.e., $P\equiv 0$ in \eqref{eq.EAS}). An intriguing discovery is the global well-posedness of the system for any $\alpha\in(0,2)$. Particularly, when $\alpha\in(0,1)$, in contrast to \eqref{eq:fractalBurgers}, the regularization effect of \eqref{eq:alignment} is strong enough to ensure global regularity for all smooth initial data. Extensions are made in \cite{kiselev2018global,miao2021global} considering general singular communication protocols.

The understanding of the theory for the pressure-less Euler-alignment system becomes more intricate in multi-dimensions. Global well-posedness has been established mostly for small initial data perturbed around $\rho\equiv1$.
See the recent result of Shvydkoy \cite{shvydkoy2019global} considering the periodic domain $\mathbb{T}^N$ with $\alpha\in(0,2)$, and Danchin et al. \cite{danchin2019regular} on the whole space $\R^N$ with $\alpha\in(1,2)$. Global regularity for generic initial data is available only for uni-directional flows \cite{lear2021unidirectional,lear2023global,li2023global} with $\alpha\in[1,2)$.

We are interested in examining the Euler-alignment system \eqref{eq.EAS} with pressure \eqref{eq:pressure}. The introduction of pressure disrupts certain conserved quantities pivotal for establishing a global well-posedness theory in the pressure-less system.

In the context of one-dimensional torus $\mathbb{T}$, Constantin et al. \cite{constantin2020entropy} established a global well-posedness theory for sufficiently large $\alpha\in(\frac53,2)$. Their result requires an additional strong local dissipation term of the type $(\mu(\rho)u_x)_x$ to be incorporated into the system.

In higher dimensions, existing global regularity results generally rely on imposing smallness conditions on the initial data. 
Chen et al. \cite{chen2021global} established global well-posedness for smooth initial data under smallness assumptions in the spatial domain $\mathbb{T}^N$ for $\alpha\in(0,2)$.
For the Euler-alignment system \eqref{eq.EAS}-\eqref{eq:pressure} in $\R^N$, global well-posedness has been studied by the authors \cite{bai2024global}, considering small initial data in an appropriate critical Besov space: 
\[\rho-1\in\widetilde{B}^{\frac{N}{2}+1-\alpha,\frac{N}{2}}\triangleq\dot{B}_{2,1}^{\frac{N}{2}+1-\alpha}\cap \dot{B}_{2,1}^{\frac{N}{2}} \quad\text{and}\quad u\in\dot{B}_{2,1}^{\frac{N}{2}+1-\alpha},\]
when $\alpha\in(1,2)$. The scaling invariant hybrid Besov space $\widetilde{B}^{\frac{N}{2}+1-\alpha,\frac{N}{2}}$ (with $\alpha=2$) was first introduced by Danchin \cite{danchin2000global} on the barotropic compressible Navier-Stokes system, wherein the alignment term is replaced by  
\begin{equation}\label{eq:baroNS}
\mathcal{D}_{\mathrm{loc}}^{2}(u) = \mu_1 \Delta u +\mu_2 \nabla\Div u,\quad\text{with}\quad \mu_1>0,\,\mu_1+\mu_2>0.
\end{equation}

As a companion to \cite{bai2024global}, our first result focuses on the global well-posedness of the Euler-alignment system \eqref{eq.EAS} in $\R^N$, considering the parameter range $\alpha\in(0,1]$.

\begin{theorem}[Global well-posedness]\label{thm:global}
Let $s> \frac{N}2+1$.
Consider the Euler-alignment system \eqref{eq.EAS} with pressure \eqref{eq:pressure}, alignment interaction \eqref{eq:alignment} with $0<\alpha\le 1$, and initial data $(\rho_0-1,u_0)\in H^s(\R^N)$.
There exists a small constant $\varepsilon>0$ such that if
\begin{equation}\label{eq:cond0}
  \Vert (\rho_0-1,u_0)\Vert_{H^s}<\varepsilon,
\end{equation}
then the Euler-alignment system \eqref{eq.EAS} has a global unique strong solution $(\rho,u)$ such that
\begin{equation}\label{eq:rho-u-bdd0}
\begin{split}
  & \rho>0\;\; \textrm{in}\;\;\R_+\times \R^N,\quad (\rho-1,u) \in C_b(\R_+; H^s(\R^N)), \\
  & \Lambda^{1-\frac{\alpha}{2}}\rho\in L^2(\R_+; H^{s+\alpha-1}(\R^N)), \quad
  \Lambda^{\frac{\alpha}{2}}u\in L^2(\R_+;H^s(\R^N)).
\end{split}
\end{equation}
\end{theorem}

\begin{remark}
A notable distinction in the case where $\alpha<1$ is the absence of the scaling invariant critical space that was utilized in \cite{bai2024global}. It is unclear how to formulate a scaling critical Banach space within our setup. Consequently, we opt to work with general Sobolev spaces that are scaling subcritical. Our result extends the work in \cite{chen2021global} from the periodic domain $\mathbb{T}^N$ to the whole space $\R^N$.
\end{remark}

Our next result focuses on the asymptotic behavior of the solution to the Euler-alignment system \eqref{eq.EAS}. For initial data satisfying \eqref{eq:cond0}, it is known that the solution $(\rho,u)$ converges to the steady state $(1,0)$. In particular, the limit
\[u(t,x)\to\bar{u}\]
illustrates the phenomenon of \emph{velocity alignment}, a collective behavior that has been extensively studied recently \cite{tadmor2014critical,shvydkoy2017eulerian2,chen2021global,leslie2023sticky}. Given the translation invariance inherent in the system, we proceed under the assumption $\bar{u}=0$ without loss of generality.

Our goal is to quantify the convergence rate of the solution towards equilibrium. In the context of barotropic compressible Navier-Stokes system \eqref{eq:baroNS}, the asymptotic behavior has been well-studied in the literature, e.g. \cite{serrin1959uniqueness,nash1962probleme, solonnikov1977estimates,valli1982existence,matsumura1979initial,matsumura1980initial,danchin2000global}. Notably, in $\R^3$, the optimal decay rate has been demonstrated (e.g., \cite{matsumura1979initial}):
\begin{equation}\label{eq:L2-decay-CNS}
  \|(\rho -1, u)(t)\|_{L^2} \sim \langle t\rangle^{-\frac{3}{4}},
\end{equation}
where the notation $\langle t\rangle=1+t$ is used.

For the Euler-alignment system \eqref{eq.EAS}, a similar optimal decay rate was obtained in \cite{bai2024global} for $\alpha\in(1,2)$:
\begin{equation}\label{eq:L2-decay-alphage1}
  \|(\rho -1, u)(t)\|_{L^2} \sim \langle t\rangle^{-\frac{N}{2\alpha}}.
\end{equation}
The decay rate agrees with \eqref{eq:L2-decay-CNS} when we take $\alpha\to2$ and $N=3$. Moreover, the decay rate on $\|u\|_{L^2}$ matches with the rate observed for the fractal Burgers equation \eqref{eq:fractalBurgers}, as well as the fractional heat equation 
\[\partial_tu=-\mu\Lambda^\alpha u.\]
This indicates that the alignment operator $\mathcal{D}(u,\rho)$ has a similar regularization effect as $\mathcal{D}(u,1)$.

However, this similarity does \emph{not} hold when $\alpha<1$. To illustrate this difference, let's consider a special case when $\alpha\to0$. In this scenario, the alignment operator formally transforms into a local damping term:
\[\mathcal{D}(u,\rho)=-\mu\rho\big(\Lambda^\alpha u + \Lambda^\alpha((\rho-1)u)-u\Lambda^\alpha(\rho-1)\big)\to-\mu\rho u,\quad\text{as}\,\,\alpha\to0.\]
System \eqref{eq.EAS} then resembles the compressible Euler system with damping. Extensive studies in the literature \cite{wang2001pointwise, tan2013global, sideris2003long, tan2012large, chen2014time} have investigated the global well-posedness and asymptotic behavior of this system. Particularly in $\R^3$, optimal decay rates have been observed (e.g., \cite{tan2012large}):
\begin{equation}\label{eq:L2-decay-ED}
\|\rho(t) -1\|_{L^2} \sim \langle t\rangle^{-\frac{3}{4}},\quad \|u(t)\|_{L^2} \sim \langle t\rangle^{-\frac{5}{4}}.
\end{equation}
The rate notably differs from \eqref{eq:L2-decay-alphage1}. Indeed, when $\alpha\to0$, the fractional heat equation becomes $\partial_tu=-\mu u$, leading to an exponential decay over time. This clearly demonstrates that the behavior of the alignment operators $\mathcal{D}(u,\rho)$ does not mirror that of the fractional Laplacian $\mathcal{D}(u,1)$,  when $\alpha$ is close to zero.

We present the following result concerning the asymptotic behavior of the solution to the Euler-alignment system \eqref{eq.EAS}, with optimal decay rates, for $\alpha\in(0,1]$.

\begin{theorem}[Asymptotic behavior]\label{thm:decay}
Let $N\ge 2$, $0<\alpha\le 1$ and $s>\frac{N}2+1$.
Suppose that $(\rho_0-1,u_0) \in H^s \cap L^1(\R^N)$ and $(\rho,u)$ is a global solution of the Euler-alignment system \eqref{eq.EAS} such that
\begin{align*}
  \sup_{t\ge0}\big(\Vert \rho(t)-1\Vert_{H^s}+\Vert u(t)\Vert_{H^s}\big)< +\infty.
\end{align*}
Then we have the following estimates for the solution $(\rho,u)$:
\begin{enumerate}
\item Decay estimates in $L^2$ norm:
\begin{equation}\label{eq:decay}
 \Vert \rho(t)-1\Vert_{L^2}\lesssim \langle t\rangle^{-\frac{N}{2(2-\alpha)}},\quad
 \Vert u(t)\Vert_{L^2}\lesssim \langle t\rangle^{-\frac{N+2(1-\alpha)}{2(2-\alpha)}}.
\end{equation}
\item Decay estimates in other norms:
\begin{equation}\label{eq:decay2}
 \Vert (\rho-1,u)(t)\Vert_{L^\infty}\lesssim\langle t\rangle^{-\frac{N}{2-\alpha}},\quad
 \Vert \nabla \rho(t)\Vert_{L^2}\lesssim\langle t\rangle^{-\frac{N+2}{2(2-\alpha)}},\quad
 \Vert \Lambda^\alpha u(t)\Vert_{L^2}\lesssim\langle t\rangle^{-\frac{N+2}{2(2-\alpha)}}.
\end{equation}
\item Decay estimates for the incompressible part $\mathbb{P}u=u - \nabla \Delta^{-1} \Div u$:
\begin{equation}\label{eq:decay.incomp}
  \Vert \mathbb{P}u(t)\Vert_{L^2}
  \lesssim \langle t\rangle^{-\frac{N}{2\alpha}},
\end{equation}
for $\alpha\in[\frac{2N}{3N+2},1]$.
\item Lower bounds: for $\alpha\in(0,1)$, assuming $\int_{\R^N} \big(\rho_0-1\big)\dd x\neq 0$ and $\int_{\R^N} \rho_0 u_0\dd x\neq 0$, we have
\begin{equation}\label{eq:decay.lower}
  \Vert \rho(t)-1\Vert_{L^2}\gtrsim\langle t\rangle^{-\frac{N}{2(2-\alpha)}},\quad
  \Vert u(t)\Vert_{L^2}\gtrsim\langle t\rangle^{-\frac{N+2(1-\alpha)}{2(2-\alpha)}}.
\end{equation}
\end{enumerate}
\end{theorem}
 
\begin{remark}
The decay rates the we obtained in \eqref{eq:decay} are \emph{optimal}. Indeed, the lower bounds offered in \eqref{eq:decay.lower} has the same rate. We have 
\[
  \|\rho(t) -1\|_{L^2} \sim \langle t\rangle^{-r_1(\alpha)}\quad\text{and}\quad
  \|u(t)\|_{L^2} \sim \langle t\rangle^{-r_2(\alpha)},
\]
with the rates
\[
 r_1(\alpha)=\tfrac{N}{2(2-\alpha)}\quad\text{and}\quad
 r_2(\alpha)=\tfrac{N+2(1-\alpha)}{2(2-\alpha)},\quad\text{for}~\alpha\in(0,1].
\]
The result is in companion with the optimal decay rates \eqref{eq:L2-decay-ED} obtained in \cite[Theorem 1.2]{bai2024global}, where
\[
 r_1(\alpha)=r_2(\alpha)=\tfrac{N}{2\alpha},\quad\text{for}~\alpha\in(1,2).
\]
Figure \ref{pic:decay} illustrates the optimal decay rates $r_1(\alpha)$ and $r_2(\alpha)$ for the whole range $\alpha\in(0,2)$.
In particular, when $\alpha\to0$, we recover the optimal decay rates for the compressible Euler system with damping \eqref{eq:L2-decay-ED}; and when $\alpha\to1$, the rates in \eqref{eq:decay} align with the decay rate for the fractional heat equation \eqref{eq:L2-decay-alphage1} with $\alpha=1$.
\end{remark}

\begin{figure}[htbp]
\centering
\begin{tikzpicture}
	\draw[->, thick] (0,0) -- (4.8,0) node[right]{$\alpha$};
	\draw[->, thick] (0,0) -- (0,3.5) node[above]{$r(\alpha)$};
	\node at (-.2,-.35) {0};
	\draw (2,.1) -- (2,-.1) node[below]{1};
	\draw (4,.1) -- (4,-.1) node[below]{2};
	\draw (.1,1.5) -- (-.1,1.5) node[left]{$\frac12$};
	\draw (.1,3) -- (-.1,3) node[left]{1};
	\draw[dashed] (0,3) -- (2,3) -- (2,0);
	\draw[dashed] (0,1.5) -- (4,1.5) -- (4,0);
	\draw[domain=2:4,very thick] plot (\x,{6/\x});
	\draw[domain=0:2,very thick,blue] plot (\x,{6/(4-\x)});
	\node[blue] at (1.5,1.9) {$r_1$};
	\draw[very thick,red] (0,3) -- (2,3);
	\node[red] at (1,3.3) {$r_2$};
	\node at (2,-1) {$N=2$};
\end{tikzpicture}
\begin{tikzpicture}
	\draw[->, thick] (0,0) -- (4.8,0) node[right]{$\alpha$};
	\draw[->, thick] (0,0) -- (0,3.5) node[above]{$r(\alpha)$};
	\node at (-.2,-.35) {0};
	\draw (2,.1) -- (2,-.1) node[below]{1};
	\draw (4,.1) -- (4,-.1) node[below]{2};
	\draw (.1,1.5) -- (-.1,1.5) node[left]{$\frac{N}{4}$};
	\draw (.1,2.25) -- (-.1,2.25) node[left]{$\frac{N+2}{4}$};
	\draw (.1,3) -- (-.1,3) node[left]{$\frac{N}{2}$};
	\draw[dashed] (0,3) -- (2,3) -- (2,0);
	\draw[dashed] (0,1.5) -- (4,1.5) -- (4,0);
	\draw[domain=2:4,very thick] plot (\x,{6/\x});
	\draw[domain=0:2,very thick,blue] plot (\x,{6/(4-\x)});
	\node[blue] at (1.5,1.9) {$r_1$};
	\draw[domain=0:2,very thick,red] plot (\x,{1.5*(6-\x)/(4-\x)});
	\node[red] at (.5,2.65) {$r_2$};
	\node at (2,-1) {$N\ge 3$};
\end{tikzpicture}
\caption{The optimal decay rates.}\label{pic:decay}
\end{figure}
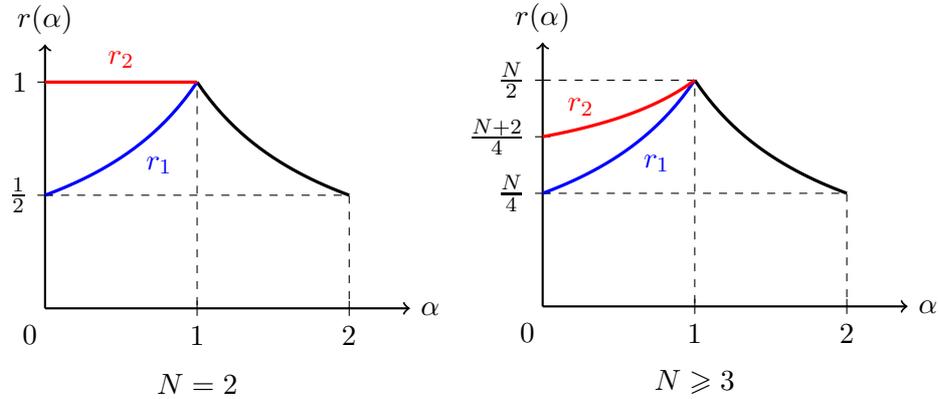

\begin{remark}
The estimate \eqref{eq:decay.incomp} reveals that the incompressible part exhibits the same decay rate as the fractional heat equation \eqref{eq:L2-decay-alphage1}. This suggests that the predominant factor contributing to the varied decay rates in the $\alpha<1$ case is the density stretching, i.e., the compressible component of the dynamics.
\end{remark}



The rest of the paper is organized as follows.
In Section \ref{Preliminary}, we provides several auxiliary lemmas to be used in the subsequent sections.
Section \ref{global_solution} is devoted to the proof of Theorem \ref{thm:global} by mainly establishing the global a priori estimates.
In Section \ref{decay}, we prove Theorem \ref{thm:decay} concerning the optimal decay estimates for the constructed global-in-time solution.

\section{Preliminary}\label{Preliminary}
In this section we compile several useful tools including some product, commutator and composition estimates, and the decay estimates associated with the fractional heat operator, and an inequality of a convolutional integral.

We first recall the following fractional Leibniz rule (see e.g. \cite[Theorem 1.2]{li2019kato}).
\begin{lemma}\label{lem:Leibniz}
Let $s>0$, $1\le p<\infty$ and $1<p_1,p_2,q_1,q_2\le \infty$ with $\frac{1}{p}= \frac{1}{p_1} + \frac{1}{p_2} = \frac{1}{q_1} + \frac{1}{q_2}$.
Then there exists a constant $C=C(s,p,p_1,p_2,q_1,q_2,N)$ such that
\begin{equation}\label{Leibnitz}
  \|\Lambda^s (uv)\|_{L^p}
  \le C \big(\|\Lambda^{s} u\|_{L^{p_1}}\|v\|_{L^{p_2}}+\|\Lambda^s v\|_{L^{q_1}}\|u\|_{L^{q_2}}\big).
\end{equation}
\end{lemma}

The following estimates can be derived by Bony's decomposition. See \cite[Equation (2.29), Theorems 2.47 and 2.52]{bahouri2011fourier}.
\begin{lemma}\label{lem:prod-es}
  Let $u: \R^N \rightarrow \R^N$ be a vector field of $\R^N$
and $f:\R^N\rightarrow \R$ be a function. Then for every $s\geq 0$, we have
\begin{align*}
  \|u\cdot \nabla u\|_{\dot H^s} \leq C \|u\|_{L^\infty} \|\nabla u\|_{\dot H^s},
\end{align*}
and
\begin{align*}
  \|u\cdot\nabla f\|_{\dot H^s} + \|(\Div u)\,f\|_{\dot H^s}
  \leq C \big( \|u\|_{L^\infty} \| f\|_{\dot H^{s+1}}
  + \|u\|_{\dot H^{s+1}} \|f\|_{L^\infty}\big).
\end{align*}
\end{lemma}

The next lemma provides some commutator estimates that can be found in \cite[Theorem 5.1]{li2019kato}.
These estimates will be useful in dealing with the alignment term.
\begin{lemma}[Commutator estimates]\label{lem:commutator}
For $s>1$, there exists a constant $C=C(s, N)$ such that
\begin{align}
  &\|\Lambda^{s}(fg)-f\Lambda^{s}g\|_{L^2}\leq C(\|\Lambda^s f\|_{L^2}\|g\|_{L^\infty}+\|\nabla f\|_{L^\infty}\|\Lambda^{s-1}g\|_{L^2}), \label{comm21}\\
  &\|\Lambda^s(fg)-f\Lambda^sg-g\Lambda^sf\|_{L^2}\leq C(\|\Lambda^{s-1} f\|_{L^2}\|\nabla g\|_{L^\infty} + \|\nabla f\|_{L^\infty} \|\Lambda^{s-1} g\|_{L^2}). \label{comm3}
\end{align}
For $s\in(0,1]$,  there exists a constant $C=C(s, N)$ such that
\begin{align}
  &\label{comm22}\|\Lambda^s(fg)-f\Lambda^sg\|_{L^2}\leq C\|\Lambda^{s}f\|_{L^2}\|g\|_{L^\infty},\\
  &\label{comm23}\|\Lambda^s(fg)-f\Lambda^sg-g\Lambda^sf\|_{L^2}\leq C\|\Lambda^{s/2}f\|_{L^\infty}\|\Lambda^{s/2}g\|_{L^2},\\
  &\label{comm24}\|\Lambda^s(fg)-f\Lambda^sg\|_{L^1}\leq C\|\Lambda^{s}f\|_{L^2}\|g\|_{L^2}.
\end{align}
\end{lemma}

The following lemma state some composition estimates.
Denote $\lceil s\rceil$ to be the smallest integer that is bigger than or equal to $s$.
\begin{lemma}[Composition estimates]\label{Lem:composite}
Let $u\in L^\infty(\R^N)$ and $F\in C^{\infty}(\textnormal{Range}(u))$ such that $F(0) = 0$.
Then the following statements hold.
\begin{itemize}
\item[(1)] If $u\in L^2(\R^N)$, we have
\begin{equation*}
  \Vert {F(u)}\Vert_{L^2}\le \Vert F'\Vert_{L^\infty(\textnormal{Range}(u))} \Vert {u}\Vert_{L^2}.
\end{equation*}
\item[(2)] If $u\in \dot{H}^s(\R^N)$ with $s>0$, there exist a positive constant $C$ depend on $\Vert F\Vert_{C^{\lceil s\rceil}(\textnormal{Range}(u))}$ and $\Vert u\Vert_{L^\infty}$ such that
\begin{equation*}
  \Vert F(u)\Vert_{\dot{H}^s}\le C \Vert u\Vert_{\dot{H}^s}.
\end{equation*}
\item[(3)] If $u\in H^{s_1}(\R^N)$ with $s_1>\frac{N}2$, we have
\begin{equation*}
  \Vert \widehat{F(u)}\Vert_{L^1}\le C \Vert u\Vert_{H^{s_1}}\Vert \widehat{u}\Vert_{L^1},
\end{equation*}
where $C$ depends only on $\Vert F\Vert_{C^{\lceil s_1\rceil+1}(\textnormal{Range}(u))}$ and $\Vert u\Vert_{L^\infty}$.
\end{itemize}
\end{lemma}

\begin{proof}[Proof of Lemma \ref{Lem:composite}]
From $F(0)=0$, we have
\begin{align*}
  F(u)=uG(u),\quad \textrm{with}\quad G(u)=\int_0^1F^\prime(\theta u)\dd \theta.
\end{align*}
Then H\"older's inequality gives the statement (1).
We can exactly obtain (2) from \cite[Lemma 2.1]{lee2022sharp}.
By Young's inequality and the continuous embedding, we see that
\begin{align*}
  \Vert \widehat{F(u)}\Vert_{L^1}
  \le\Vert \widehat{u}\Vert_{L^1}\Vert \widehat{G(u)}\Vert_{L^1}\le C\Vert \widehat{u}\Vert_{L^1}\Vert G(u)\Vert_{H^{s_1}}.
\end{align*}
Thus the statement (3) follows from (1) and (2).
\end{proof}

The next two lemmas provide some explicit decay estimates associated with the fractional heat semigroup operator.
\begin{lemma}\label{lem:heat.decay}
Let $\nu>0$, $\beta>0$, $s\ge 0$ and $f\in \dot{H}^s(\R^N)\cap L^1(\R^N)$.
Let $A(D)$ be a smooth homogeneous multiplier operator of degree $0$ (e.g. $A(D)=\nabla \Delta^{-1}\Div$ or $\Lambda^{-1}\Div$).
Then
\begin{align}\label{eq:heat.decay}
  \Vert e^{-\nu t\Lambda^{\beta}}A(D)f\Vert_{\dot{H}^s(\R^N)}\le C\langle t\rangle^{-\frac{N+2s}{2\beta}}
  \Vert f\Vert_{\dot{H}^s\cap L^1(\R^N)}.
\end{align}
\end{lemma}

\begin{proof}[Proof of Lemma \ref{lem:heat.decay}]
For $t\le 1$, we have
\begin{equation}\label{eq.Heat.D1}
\begin{aligned}
  \Vert e^{-\nu t\Lambda^{\beta}}A(D)f\Vert_{\dot{H}^s}
  \le C\Vert e^{-\nu t\vert\xi\vert^{\beta}}\vert\xi\vert^s\widehat{f}\Vert_{L^2}
  \le \Vert f\Vert_{\dot{H}^s}
  \le C\langle t\rangle^{-\frac{N+2s}{2\beta}}\Vert f\Vert_{\dot{H}^s}.
\end{aligned}
\end{equation}
For $t>1$, we get
\begin{equation}\label{eq.Heat.D2}
\begin{aligned}
  \Vert e^{-\nu t\Lambda^{\beta}}A(D)f\Vert_{\dot{H}^s}
  &\le C\Vert e^{-\nu t\vert\xi\vert^{\beta}}\vert\xi\vert^s\widehat{f}\Vert_{L^2}
  \le C\Vert\vert\xi\vert^s e^{-\nu t\vert\xi\vert^{\beta}}\Vert_{L^2}\Vert \widehat{f}\Vert_{L^\infty}
  \le Ct^{-\frac{N+2s}{2\beta}}\Vert f\Vert_{L^1}.
\end{aligned}
\end{equation}
Combining \eqref{eq.Heat.D1} and \eqref{eq.Heat.D2} leads to the desired result.
\end{proof}

\begin{lemma}\label{lem:heat.decay.L1}
Let $\nu>0$ and $f,\widehat{f}\in L^1(\R^N)$.
Then
\begin{align}
  \Vert e^{-\nu t\vert\xi\vert^{\beta}}\widehat{f}\Vert_{L^1(\R^N)}\le C\langle t\rangle^{-\frac{N}{\beta}}
  (\Vert f\Vert_{L^1(\R^N)}+\Vert \widehat{f}\Vert_{L^1(\R^N)}).
\end{align}
\end{lemma}

\begin{proof}[Proof of Lemma \ref{lem:heat.decay.L1}]
For $t\le 1$, it is obvious to see that $\|e^{-\nu t|\xi|^\beta}\widehat{f}\|_{L^1} \leq C \|\widehat{f}\|_{L^1}$.
For $t>1$, we have
\begin{equation*}
\begin{aligned}
  \Vert e^{-\nu t\vert\xi\vert^{\beta}}\widehat{f}\Vert_{L^1}
  \le C\Vert e^{-\nu t\vert\xi\vert^{\beta}}\Vert_{L^1}\Vert \widehat{f}\Vert_{L^\infty}
  \le Ct^{-\frac{N}{\beta}}\Vert f\Vert_{L^1}.
\end{aligned}
\end{equation*}
The desired inequality follows by combining both cases $t\le 1$ and $t>1$.
\end{proof}

Finally we state an elementary inequality on a convolutional integral,
which will be repeatedly used in the proof of asymptotic behavior.
For the proof, it can be directly deduced by \cite[Proposition 4.5]{strain2010asymptotic}, thus we omit the details.
\begin{lemma}\label{lem:convolution}
Let $\alpha_1,\alpha_2\ge0$.
Then we have
\begin{align}
  \int_0^t\langle t-\tau\rangle^{-\alpha_1}\langle \tau\rangle^{-\alpha_2}\dd \tau
  \le C
  \begin{cases}
  \langle t\rangle^{-\min\{\alpha_1,\alpha_2\}},\quad&\max\{\alpha_1,\alpha_2\}>1,\\
  \langle t\rangle^{-\min\{\alpha_1,\alpha_2\}}\log(1+\langle t\rangle),\quad&\max\{\alpha_1,\alpha_2\}=1,\\
  \langle t\rangle^{1-\alpha_1-\alpha_2},\quad&\max\{\alpha_1,\alpha_2\}<1.
  \end{cases}
\end{align}
\end{lemma}


\section{Global well-posedness}\label{global_solution}
In this section, we establish the global well-posedness result of the Euler-alignment system \eqref{eq.EAS} with small initial data.

First we shall provide a spectral analysis for the linearized system in the first subsection.
Then, based on the spectral result, we establish the global \emph{a priori} estimates of the strong solution.
Furthermore, we prove Theorem \ref{thm:global} concerning the existence and uniqueness result with small initial data.

Denote $a \triangleq \rho-1$ and $a_0 \triangleq \rho_0-1$. Then the Euler-alignment system \eqref{eq.EAS} recasts
\begin{equation}\label{eq.EAS.a}
\left\{
\begin{aligned}
  &\partial_t a +\Div u=-\Div(au),\\
  &\partial_t u+\mu\Lambda^{\alpha}u+{\gamma}\nabla a=\mu \big(u\Lambda^{\alpha}a-\Lambda^{\alpha}(au)\big)-u\cdot\nabla u-{\gamma}(\rho^{\gamma-2}-1)\nabla a,\\
  &(a,u)|_{t=0}=(a_0,u_0).
\end{aligned}
\right.
\end{equation}

\subsection{Spectral analysis of the linearized system}
Denote by
\begin{align}\label{def:v-P}
  v \triangleq \Lambda^{-1}\Div u, \quad \textrm{and}\quad
  \mathbb{P}\triangleq \mathrm{Id} - \nabla \Delta^{-1} \mathrm{div},
\end{align}
then $u = - \nabla \Lambda^{-1} v + \mathbb{P} u $.
In light of \eqref{eq.EAS.a}, $(a,v,\mathbb{P}u)$ satisfies the following system:
\begin{equation}\label{eq.lineareq4au}
\begin{cases}
  \partial_t a + \Lambda v=F, \\
  \partial_t v + \mu\Lambda^\alpha v-{\gamma}\Lambda a=G, \\
  \partial_t \mathbb{P}u + \mu \Lambda^\alpha \mathbb{P}u = H,
\end{cases}
\end{equation}
where $(F,G,H)$ is given by
\begin{align}
  & F \triangleq -\Div(au), \label{def:F} \\
  & G \triangleq \mu \Lambda^{-1}\Div\big(u\Lambda^{\alpha}a-\Lambda^{\alpha}(au)\big)
  -\Lambda^{-1}\Div(u\cdot\nabla u)-{\gamma}\Lambda^{-1}\Div\big((\rho^{\gamma-2}-1)\nabla a\big), \label{def:G} \\
  & H \triangleq \mu \mathbb{P}\big(u\Lambda^{\alpha}a-\Lambda^{\alpha}(au)\big)
  - \mathbb{P}(u\cdot\nabla u) . \label{def:H}
\end{align}
Clearly, $\mathbb{P}u$ has the $\alpha$-order fractional dissipation effect. Next let us consider the situation of $(a,v)$.
By taking Fourier transform with respect to $x$-variable, we see that
\begin{equation}\label{eq.lineareq4au1}
  \partial_t\begin{pmatrix}\widehat{a}\\\widehat{v}\end{pmatrix}
  -A\begin{pmatrix}\widehat{a}\\\widehat{v}\end{pmatrix} =
  \begin{pmatrix}\widehat{F}\\\widehat{G}\end{pmatrix}
  \quad\text{with}\quad
  A\triangleq \begin{pmatrix}0&-\vert \xi\vert\\{\gamma}\vert\xi\vert&-\mu\vert\xi\vert^{\alpha}\end{pmatrix}.
\end{equation}
It is easy to check that the eigenvalues of matrix $A$ are
\begin{align}\label{def:lamb-pm}
  \lambda_\pm
  \triangleq \frac{-\mu\vert\xi\vert^\alpha\mp\sqrt{\mu^2\vert\xi\vert^{2\alpha}-4\gamma\vert \xi\vert^2}}{2}
  =\frac{-\mu\vert\xi\vert^\alpha}{2}\Big{(}1\pm\sqrt{1-\tfrac{4\gamma}{\mu^2}\vert \xi\vert^{2-2\alpha}}\Big{)}.
\end{align}
We divide the discussion into three cases according to the value of $\alpha\in (0,2)$.

\noindent\textbf{Case 1:} $1<\alpha<2$.
\begin{itemize}
\item For the low-frequency regime $\vert\xi\vert^{\alpha-1}<2\sqrt\gamma/\mu$, the eigenvalues can be rewritten as
\begin{align*}
  \lambda_{\pm}=\frac{-\mu\vert\xi\vert^{\alpha}}{2}\Big{(}1\pm{i}\sqrt{\tfrac{4\gamma}{\mu^2}\vert \xi\vert^{2-2\alpha}-1}\Big{)}.
\end{align*}
Hence, we expect $a$ and $v$ to have parabolic damping.
\item For the high-frequency regime $\vert\xi\vert^{\alpha-1}> 2\sqrt\gamma/\mu$, the eigenvalues become
\begin{align*}
  &\lambda_+ = \frac{-\mu\vert\xi\vert^{\alpha}}{2}
  \Big{(}1+\sqrt{1-\tfrac{4\gamma}{\mu^2}\vert \xi\vert^{2-2\alpha}}\Big{)}\sim -\mu\vert\xi\vert^{\alpha},\\
  &\lambda_{-}
  =\frac{-\vert\xi\vert^{\alpha}}{2}\frac{{{4\gamma}\vert \xi\vert^{2-2\alpha}}}{\mu+\sqrt{\mu^2-{{4\gamma}\vert \xi\vert^{2-2\alpha}}}}\sim -\frac{\gamma}{\mu}\vert\xi\vert^{2-\alpha}.
\end{align*}
We expect that $a$ has the fractional dissipation of order $2-\alpha$
and $v$ has the fractional dissipation of order $\alpha$.
\end{itemize}

\noindent\textbf{Case 2:} $\alpha=1$.
\begin{itemize}
\item For $\sqrt\gamma/\mu\le 1/2$, the eigenvalues are
\begin{align*}
  \lambda_\pm = \frac{-\mu\vert\xi\vert}{2}\Big{(}1\pm\sqrt{1-\tfrac{4\gamma}{\mu^2}}\Big{)}\sim -\vert\xi\vert.
\end{align*}
Thus the parabolic damping of $a$ and $v$ is expected.
\item For $\sqrt\gamma/\mu>1/2$, the eigenvalues read as
\begin{align*}
  \lambda_{\pm}=\frac{-\mu\vert\xi\vert}{2}\Big{(}1\pm{i}\sqrt{\tfrac{4\gamma}{\mu^2}-1}\Big{)}.
\end{align*}
We expect that $(a,v)$ will show the one-order fractional dissipation effect.
\end{itemize}

\noindent\textbf{Case 3:} $0<\alpha<1$.
\begin{itemize}
  \item For low frequencies $\vert\xi\vert^{1-\alpha}<\mu/(2\sqrt\gamma)$, the eigenvalues are
\begin{equation}\label{eq:lamb-pm-low}
\begin{split}
  &\lambda_{+}=\frac{-\mu\vert\xi\vert^{\alpha}}{2}\Big{(}1+\sqrt{1-\tfrac{4\gamma}{\mu^2}\vert \xi\vert^{2-2\alpha}}\Big{)}\sim -\mu\vert\xi\vert^{\alpha},\\
  &\lambda_{-}
  =\frac{-\vert\xi\vert^{\alpha}}{2}\frac{{4\gamma\vert \xi
  \vert^{2-2\alpha}}}{\mu+\sqrt{\mu^2-{4\gamma}\vert \xi\vert^{2-2\alpha}}}
  \sim -\frac{\gamma}{\mu}\vert\xi\vert^{2-\alpha}.
\end{split}
\end{equation}
Thus we expect that $a$ enjoys the $2-\alpha$ order fractional dissipation and $v$ has the $\alpha$-order fractional dissipation.
\item For high frequencies $\vert\xi\vert^{1-\alpha}>\mu/(2\sqrt\gamma)$, the eigenvalues read as
\begin{align}\label{eq:lamb-pm-high}
  \lambda_{\pm}=\frac{-\mu\vert\xi\vert^{\alpha}}{2}\Big{(}1\pm{i}\sqrt{\tfrac{4\gamma}{\mu^2}
  \vert \xi\vert^{2-2\alpha}-1}\Big{)}.
\end{align}
We expect that $a$ and $v$ have the parabolic damping.
\end{itemize}

In short, based on the above analysis, we may expect that $a$ (or $u$) behaves the $2-\alpha$ (or $\alpha$)
order fractional dissipation effect in the $L^2$ estimate and $(a,u)$ enjoys the $\alpha$-order dissipation
in the $\dot{H}^s$ estimate when $0<\alpha \le 1$.

\subsection{A priori estimates}
In this subsection, we establish the global \emph{a priori} estimates for the smooth solution of system \eqref{eq.EAS.a}
with small initial data.
\begin{proposition}\label{Prop:priori.au}
Let $\alpha\in(0,1]$ and $s>\frac{N}2+1-\frac{\alpha}{2}$.
Assume that $(a,u)$ is a smooth solution of system \eqref{eq.EAS.a}.
There exists a constant $\varepsilon>0$ such that if
\begin{align}\label{eq.cond0.au}
  \Vert (a_0,u_0)\Vert_{H^s}\le\varepsilon,
\end{align}
then for every $T>0$ we have
\begin{equation}\label{eq.priori.au}
\begin{aligned}
  &\Vert (a,u)\Vert^{2}_{L^\infty_T(H^s)}
  +\int_0^T\left(\Vert\Lambda^{1-\frac{\alpha}{2}}a\Vert_{H^{s+\alpha-1}}^2
  +\Vert\Lambda^{\frac{\alpha}{2}} u\Vert^{2}_{H^s}\right)\dd t
  \le C\Vert (a_0,u_0)\Vert^{2}_{H^s}.
\end{aligned}
\end{equation}
\end{proposition}

To prove Proposition \ref{Prop:priori.au}, controlling complicated terms such as ${\gamma}(\rho^{\gamma-2}-1)\nabla a$ in a direct energy estimate for the system \eqref{eq.EAS.a} poses a challenge. To overcome this difficulty, we adopt the approach from \cite{sideris2003long} and introduce an auxiliary quantity $\sigma$, defined as follows:\begin{align}\label{sigma}
  \sigma \triangleq \left\{\begin{array}{ll}\ln \rho & \textrm{for}\;\;\gamma=1,\\
  \frac{2\sqrt{\gamma}}{\gamma-1}(\rho^{\frac{\gamma-1}{2}}-1)& \textrm{for}\;\;\gamma>1,
  \end{array}\right.
  \quad\text{and}\quad
  \sigma_0 \triangleq \left\{\begin{array}{ll}\ln \rho_0 & \textrm{for}\;\; \gamma=1,\\
  \frac{2\sqrt{\gamma}}{\gamma-1}(\rho_0^{\frac{\gamma-1}{2}}-1)& \textrm{for}\;\; \gamma>1.
  \end{array}\right.
\end{align}
Conversely, $a=\rho-1$ can be expressed by
\begin{align}\label{eq:a-exp}
  a =
  \left\{\begin{array}{ll}
  e^\sigma -1 & \quad \textrm{for}\;\; \gamma=1,\\
  \Big(\frac{\gamma-1}{2\sqrt{\gamma}}\sigma+1\Big)^{\frac{2}{\gamma-1}} -1 & \quad \textrm{for}\;\;\gamma>1.
\end{array}\right.
\end{align}
If $\rho>0$, the system \eqref{eq.EAS} (or \eqref{eq.EAS.a}, equivalently) is transformed into the
following system
\begin{equation}\label{eq.EAS.sigma}
\left\{
\begin{aligned}
  &\partial_t \sigma +\sqrt{\gamma}\Div u=-u\cdot\nabla\sigma-\tfrac{\gamma-1}{2}\sigma\Div u,\\
  &\partial_t u+\mu\Lambda^{\alpha}u+\sqrt{\gamma}\nabla \sigma=\mu \big(u\Lambda^{\alpha}a-\Lambda^{\alpha}(au)\big)
  -u\cdot\nabla u-\tfrac{\gamma-1}{2}\sigma\nabla \sigma,\\
  &(\sigma,u)|_{t=0}=(\sigma_0,u_0).
\end{aligned}
\right.
\end{equation}

For the smooth solution $(\sigma,u)$ of system \eqref{eq.EAS.sigma}, we can get the following \emph{a priori} estimate.
\begin{proposition}\label{Prop:priori}
Let $\alpha\in(0,1]$ and $s>\frac{N}2+1-\frac{\alpha}{2}$.
Assume that $(\sigma,u)$ is a smooth solution of system \eqref{eq.EAS.sigma}.
Then there exists a constant $\varepsilon^\prime>0$ such that if
\begin{align}\label{eq.cond0.su}
  \Vert (\sigma_0,u_0)\Vert_{H^s} \le \varepsilon^\prime,
\end{align}
then it holds
\begin{equation}\label{eq.Prop1}
\begin{aligned}
  & \Vert (\sigma,u)\Vert^{2}_{L^\infty_T(H^s)}
  +\int_0^T \Big(\Vert\Lambda^{\frac{\alpha}{2}} u\Vert^{2}_{H^s}
  +\Vert\Lambda^{1-\frac{\alpha}{2}}\sigma\Vert_{H^{s+\alpha-1}}^2\Big)\dd t
  \le C\Vert (\sigma_0,u_0)\Vert^{2}_{H^s}.
\end{aligned}
\end{equation}
\end{proposition}

With Proposition \ref{Prop:priori} at our disposal, we can immediately show the proof of Proposition \ref{Prop:priori.au}.
\begin{proof}[Proof of Proposition \ref{Prop:priori.au}]
By letting $0<\varepsilon<1$ small enough, the initial assumption \eqref{eq.cond0.au}
and Lemma \ref{Lem:composite} ensure \eqref{eq.cond0.su}.
Then according to Proposition \ref{Prop:priori}, we obtain the \emph{a priori} estimate \eqref{eq.Prop1} of $(\sigma,u)$.
For small enough $\varepsilon^\prime$, we see that
\begin{align*}
  \Vert \sigma\Vert^2_{L^\infty_T(L^\infty)} \le C \Vert \sigma\Vert^2_{L^\infty_T(H^s)}
  < \tfrac{2\sqrt{\gamma}}{\gamma-1}.
\end{align*}
Hence, using \eqref{eq:a-exp} and Lemma \ref{Lem:composite} again, we get the desired estimate \eqref{eq.priori.au}.
\end{proof}

Now we turn to the proof of Proposition \ref{Prop:priori}.
\begin{proof}[Proof of Proposition \ref{Prop:priori}]
We divide the proof into three steps.
\vskip1mm

\noindent \textbf{Step 1}: $L^2$-estimate. We shall prove that there exist
two positive constants $C=C(\alpha,s,N,\gamma,\mu)$ and
$\widetilde{C}_1 = \widetilde{C}_1(\alpha,s,N)$
such that for every $s> \frac{N}{2} +1-\frac{\alpha}{2}$,
\begin{align}\label{eq.D.Lem1}
  &\Vert (\sigma,u)\Vert^{2}_{L^\infty_T(L^2)}
  +\int_0^T \Big( \delta_1 \sqrt{\gamma}\Vert\Lambda^{1-\frac{\alpha}{2}}\sigma\Vert_{L^2}^2
  + \mu \Vert\Lambda^{\frac{\alpha}{2}} u\Vert^2_{L^2}\Big) \dd t
  \le C\Vert (\sigma_0,u_0)\Vert^{2}_{H^s} +  \widetilde{C}_1 \delta_1
  \Vert \Lambda^s\sigma\Vert^{2}_{L^\infty_T(L^2)}  \nonumber \\
  & + \widetilde{C}_1 \delta_1  \int_0^T
  \Vert\Lambda^{\frac{\alpha}{2}+s} u\Vert^2_{L^2} \dd t
  + C\Vert (\sigma,u)\Vert_{L^\infty_T(H^s)}
  \int_0^T \Big(\Vert \Lambda^{1-\frac{\alpha}{2}}\sigma\Vert_{H^{s+\alpha-1}}^2
  +\Vert \Lambda^{\frac{\alpha}{2}}u\Vert_{H^s}^2\Big) \dd t,
\end{align}
where $0<\delta_1 \leq \min\{1,\frac{\mu}{\sqrt{\gamma}}\}$ is a constant chosen later.

Taking the $L^2$ inner product of the first equation of \eqref{eq.EAS.sigma}
with $\sigma$ and performing the integrations by parts, we see that
\begin{align}\label{eq.D.sigma}
  \frac{1}{2}\frac{\dd}{\dd t}\Vert \sigma\Vert^{2}_{L^2}+\sqrt{\gamma} \int_{\R^N}\sigma\Div u\,\dd x
  =\frac{2-\gamma}{2} \int_{\R^N}\sigma^2\Div u\,\dd x \triangleq I.
\end{align}
A similar $L^2$-energy argument of $u$ gives
\begin{equation}\label{eq.D.u}
\begin{aligned}
  &\frac{1}{2}\frac{\dd}{\dd t}\Vert u\Vert^{2}_{L^2} +\mu\Vert\Lambda^{\frac{\alpha}{2}} u\Vert^{2}_{L^2}
  - \sqrt{\gamma}\int_{\R^N} \sigma \Div u \,\dd x \\
  & \hspace{.1cm}= \int_{\R^N} u^2\Div u\,\dd x
  + \mu \int_{\R^N} a \big(\Lambda^{\alpha} |u|^2-u\cdot\Lambda^{\alpha}u\big)\dd x
  - \frac{\gamma-1}{4} \int_{\R^N}\sigma^2\Div u \,\dd x   \triangleq \sum_{j=1}^{3}J_j.
\end{aligned}
\end{equation}
In order to get the smoothing effect of $\sigma$,
we consider the following quantity $\int_{\R^N} u\cdot\nabla \Lambda^{-\alpha}\sigma\dd x$
(recalling the spectral analysis in low frequencies).
By taking the $L^2$ inner product of the first and second equations of \eqref{eq.EAS.sigma} with
$\Lambda^{1-\alpha} v = \Lambda^{-\alpha} \mathrm{div}\,u$
and $\nabla\Lambda^{-\alpha}\sigma$ respectively, and summing them up, we obtain
\begin{equation}\label{eq.D.mix}
\begin{aligned}
  &\quad\frac{\dd}{\dd t}\int_{\R^N} u\cdot\nabla \Lambda^{-\alpha}\sigma\dd x
  + \sqrt{\gamma} \Vert\Lambda^{1-\frac{\alpha}{2}}\sigma\Vert_{L^2}^2 -\mu\int_{\R^N} \sigma\Div u\dd x
  -\sqrt{\gamma}\Vert\Lambda^{-\frac{\alpha}{2}}\Div u\Vert_{L^2}^2\\
  &=\int_{\R^N} u\cdot\nabla\sigma\Lambda^{-\alpha}\Div u\,\dd x
  +\frac{\gamma-1}{2}\int_{\R^N}\sigma\Div u\,\Lambda^{-\alpha}\Div u\,\dd x
  -\frac{\gamma-1}{4}\int_{\R^N} \sigma^2\Lambda^{2-\alpha}\sigma\,\dd x \\
  &\quad+\mu \int_{\R^N} \big(u\Lambda^{\alpha}a-\Lambda^{\alpha}(a u)\big) \cdot \nabla \Lambda^{-\alpha}\sigma\,\dd x
  - \int_{\R^N} \big(u\cdot\nabla u\big) \cdot\nabla \Lambda^{-\alpha}\sigma\,\dd x
  \triangleq \sum_{j=1}^5 K_j.
\end{aligned}
\end{equation}
Since $s>1-\alpha$, taking advantage of H\"older's inequality and  the interpolation inequality,
we deduce that there exists a constant $C_1>0$ depending only on $\alpha,s$ such that
\begin{equation}\label{eq.D.L1}
\begin{aligned}
  \left\vert\int_{\R^N} u\cdot\nabla \Lambda^{-\alpha}\sigma\,\dd x\right\vert
  &\le \Vert u\Vert_{L^2}\Vert \nabla \Lambda^{-\alpha}\sigma\Vert_{L^2}\le C\Vert u\Vert_{L^2}
  \Vert \sigma\Vert_{L^2}^{\frac{s+\alpha-1}{s}}\Vert \Lambda^{s}\sigma\Vert_{L^2}^{\frac{1-\alpha}{s}} \\
  & \le C_1\Big(\tfrac{1}{2}\Vert u\Vert_{L^2}^2
  + \tfrac{s+\alpha-1}{2s} \Vert\sigma\Vert_{L^2}^2 + \tfrac{1-\alpha}{2s}\Vert\Lambda^{s}\sigma\Vert_{L^2}^{2}\Big),
\end{aligned}
\end{equation}
and
\begin{equation}\label{eq.D.L2}
\begin{split}
  \sqrt{\gamma}\Vert\Lambda^{-\frac{\alpha}{2}}\Div u\Vert_{L^2}^2
  & \le C \sqrt{\gamma} \Vert\Lambda^{\frac{\alpha}{2}}u\Vert_{L^2}^{\frac{2(s+\alpha-1)}{s}}
  \Vert \Lambda^{\frac{\alpha}{2}+s}u\Vert_{L^2}^{\frac{2(1-\alpha)}{s}}  \\
  & \le  C_1 \sqrt{\gamma}\Big( \tfrac{s+\alpha-1}{2s}\Vert\Lambda^{\frac{\alpha}{2}}u\Vert_{L^2}^2
  + \tfrac{1-\alpha}{2s}\Vert \Lambda^{s+\frac{\alpha}{2}}u \Vert_{L^2}^2\Big).
\end{split}
\end{equation}
We set $0< \delta_1 \leq \min\{1,\tfrac{\mu}{\sqrt{\gamma}}\}$ to be a constant chosen later and
\begin{align*}
  Y(t) \triangleq \Big( 1 + \tfrac{ \delta_1 \mu}{2\sqrt\gamma C_1}\Big) \Vert \sigma(t)\Vert^2_{L^2}
  + \Vert u(t)\Vert^2_{L^2} +  \frac{\delta_1}{C_1} \int_{\R^N} u\cdot\nabla \Lambda^{-\alpha}\sigma(t)\dd x.
\end{align*}
Thanks to \eqref{eq.D.L1}, it is obvious to get
\begin{align}\label{eq.D.Ysim}
  \tfrac{1}{2} \big(\Vert \sigma\Vert^2_{L^2} + \Vert u\Vert^2_{L^2}\big)
  - \tfrac{(1-\alpha)\delta_1}{2s}\Vert\Lambda^s\sigma\Vert^2_{L^2}
  \le Y(t) \le C \big(\Vert \sigma\Vert^{2}_{L^2}+\Vert u\Vert^{2}_{L^2}\big)
  +\tfrac{(1-\alpha)\delta_1}{2s} \Vert\Lambda^{s}\sigma\Vert^2_{L^2}.
\end{align}
Collecting the equations \eqref{eq.D.sigma}-\eqref{eq.D.mix}, and using \eqref{eq.D.L2}, we have
\begin{equation}\label{eq.D.mix1}
\begin{aligned}
  &\frac{1}{2}\frac{\dd}{\dd t}Y(t)
  +\frac{\mu}{2}\Vert\Lambda^{\frac{\alpha}{2}} u\Vert^{2}_{L^2}
  +\tfrac{ \delta_1 \sqrt{\gamma}}{2 C_1} \Vert\Lambda^{1-\frac{\alpha}{2}}\sigma\Vert_{L^2}^2
  \le C \bigg(\vert I\vert +  \sum_{j=1}^{3} |J_j| + \sum_{j=1}^5\vert K_j\vert \bigg)
  + \tfrac{\delta_1 \sqrt{\gamma}(1-\alpha)}{4s} \Vert \Lambda^{\frac{\alpha}{2}+s}u\Vert_{L^2}^2.
\end{aligned}
\end{equation}

Next we estimate the right-hand side of \eqref{eq.D.mix1} one by one.
For the terms $ I$ and $J_3$, by virtue of the $L^2$ boundedness of the Riesz transform and Lemma \ref{lem:Leibniz},
we have
\begin{align}\label{eq.D.B1}
  \vert I\vert + \vert J_3 \vert
  \le C\Vert\Lambda^{1-\frac{\alpha}{2}}\sigma^2\Vert_{L^2}\Vert\Lambda^{\frac{\alpha}{2}-1}\Div u\Vert_{L^2}
  \le C\Vert\sigma\Vert_{L^\infty} \Vert\Lambda^{1-\frac{\alpha}{2}}\sigma\Vert_{L^2}\Vert\Lambda^{\frac{\alpha}{2}} u\Vert_{L^2}.
\end{align}
A similar argument gives
\begin{align*}
  \vert J_1\vert = \Big|\int_{\R^N} u^2\Div u \,\dd x\Big|
  \le C\Vert u\Vert_{L^\infty}\Vert\Lambda^{1-\frac{\alpha}{2}}u\Vert_{L^2}\Vert\Lambda^{\frac{\alpha}{2}} u\Vert_{L^2}.
\end{align*}
For $J_2$ given by \eqref{eq.D.u}, using \eqref{eq:a-exp} and Lemmas \ref{lem:commutator} and \ref{Lem:composite},
we have
that for every $\alpha\in(0,1)$,
\begin{align*}
  |J_2|
  \le C\Vert a\Vert_{L^2}\Vert \Lambda^\alpha |u|^2 - u\cdot\Lambda^{\alpha}u\Vert_{L^2}
  \le C\Vert \sigma\Vert_{L^2}\Vert \Lambda^{\alpha}u\Vert_{L^2}\Vert u\Vert_{L^\infty},
\end{align*}
and for $\alpha=1$,
\begin{equation*}
\begin{aligned}
  |J_2|
  &= \Big| \mu \int_{\R^N} a(\Lambda^{\alpha}|u|^2-2u\cdot\Lambda^\alpha u)\,\dd x
  +\mu \int_{\R^N} \Lambda^{\frac{\alpha}{2}}(au)\cdot\Lambda^{\frac{\alpha}{2}}u\,\dd x \Big| \\
  &\le C\Vert a\Vert_{L^\infty}\Vert \Lambda^{\frac{\alpha}{2}}u\Vert_{L^2}^2
  + C \Vert u\Vert_{L^2} \Vert \Lambda^{\frac{\alpha}{2}}a\Vert_{L^\infty}
  \Vert \Lambda^{\frac{\alpha}{2}}u\Vert_{L^2} \\
  &\le C\Vert \sigma\Vert_{L^\infty}\Vert \Lambda^{\frac{\alpha}{2}}u\Vert_{L^2}^2
  + C \Vert u\Vert_{L^2} \Vert \Lambda^{\frac{\alpha}{2}}\sigma\Vert_{H^{s +\alpha -1}}
  \Vert \Lambda^{\frac{\alpha}{2}}u\Vert_{L^2},
\end{aligned}
\end{equation*}
where in the last line we have used the embedding $H^{s'+\alpha -1} \hookrightarrow L^\infty$ for every $s' > \frac{N}{2} +1-\frac\alpha 2$.
By the $L^2$-boundedness of the Riesz transform and Lemma \ref{lem:Leibniz}, we can estimate $K_1$ and $K_3$ as follows:
\begin{equation*}
\begin{aligned}
  \vert K_1\vert
  &\le C\Vert \Lambda^{\frac{\alpha}{2}}(u\,\Lambda^{-\alpha}\Div u)\Vert_{L^2}
  \Vert \nabla\Lambda^{-\frac{\alpha}{2}}\sigma\Vert_{L^2} \\
  &\le C\big(\Vert \Lambda^{\frac{\alpha}{2}}u\Vert_{L^\infty}\Vert\Lambda^{1-\alpha}u\Vert_{L^2}
  +\Vert u\Vert_{L^\infty}\Vert\Lambda^{1-\frac{\alpha}{2}} u\Vert_{L^2}\big)
  \Vert\Lambda^{1-\frac{\alpha}{2}}\sigma\Vert_{L^2},
\end{aligned}
\end{equation*}
and
\begin{align*}
  \vert K_3\vert
  \le C\Vert \Lambda^{1-\frac{\alpha}{2}}\sigma^2\Vert_{L^2}\Vert \Lambda^{1-\frac{\alpha}{2}}\sigma\Vert_{L^2}
  \le C\Vert \sigma\Vert_{L^\infty}\Vert \Lambda^{1-\frac{\alpha}{2}}\sigma\Vert_{L^2}^2.
\end{align*}
For $K_2$, by applying Lemma \ref{lem:Leibniz} again,
we find that for every $\frac{2}{3}\leq \alpha \leq 1$,
\begin{equation*}
\begin{aligned}
  \vert K_2\vert
  &\le C \Vert \Lambda^{1-\frac{\alpha}{2}}(\sigma\Lambda^{-\alpha}\Div u)\Vert_{L^2}
  \Vert\Lambda^{\frac{\alpha}{2}-1}\Div u\Vert_{L^2} \\
  &\le C \big(\Vert \Lambda^{1-\frac{\alpha}{2}}\sigma\Vert_{L^\infty}\Vert\Lambda^{1-\alpha} u\Vert_{L^2}
  +\Vert\sigma\Vert_{L^\infty}\Vert\Lambda^{2-\frac{3\alpha}{2}}u\Vert_{L^2}\big)
  \Vert\Lambda^{\frac{\alpha}{2}} u\Vert_{L^2},
\end{aligned}
\end{equation*}
and for every $0<\alpha < \frac{2}{3}$,
\begin{align*}
  \vert K_2 \vert
  & \leq C \|\sigma\|_{L^\infty} \|\Div u\|_{L^2} \|\Lambda^{-\alpha}\Div u\|_{L^2} \\
  & \leq C \|\sigma\|_{L^\infty} \|\Lambda u\|_{L^2} \|\Lambda^{1-\alpha} u\|_{L^2}.
\end{align*}
The term $K_4$ is treated in a similar way as $J_2$:
owing to Lemmas \ref{lem:commutator} and \ref{Lem:composite}, we see that for every $\alpha\in(0,1)$,
\begin{align*}
  \vert K_4\vert
  &\le C\Vert a\Vert_{L^2}\Vert\Lambda^\alpha \big(u\cdot\nabla \Lambda^{-\alpha}\sigma)-u\cdot\nabla\sigma\big)\Vert_{L^2}
  \le C\Vert a\Vert_{L^2}\Vert\Lambda^{\alpha}u\Vert_{L^2}\Vert \Lambda^{1-\alpha}\sigma\Vert_{L^\infty},
\end{align*}
and for $\alpha=1$,
\begin{equation*}
\begin{aligned}
  \vert K_4\vert
  &= \Big| \mu \int_{\R^N} (u\Lambda^{\alpha}a-\Lambda^{\frac{\alpha}{2}}(u\Lambda^{\frac{\alpha}{2}}a))\cdot\nabla \Lambda^{-\alpha}\sigma\,\dd x
  +\mu \int_{\R^N} (u\Lambda^{\frac{\alpha}{2}}a-\Lambda^{\frac{\alpha}{2}}(ua))\cdot\nabla \Lambda^{-\frac{\alpha}{2}}\sigma\dd x \Big| \\
  &\le C\Vert \Lambda^{1-\alpha}\sigma\Vert_{L^2}\Vert \Lambda^{\frac{\alpha}{2}}u\Vert_{L^2}
  \Vert\Lambda^{\frac{\alpha}{2}}a\Vert_{L^\infty}+C\Vert \Lambda^{1-\frac{\alpha}{2}}\sigma\Vert_{L^2}\Vert \Lambda^{\frac{\alpha}{2}}u\Vert_{L^2} \Vert a\Vert_{L^\infty}  \\
  &\le C\Vert \Lambda^{1-\alpha}\sigma\Vert_{L^2} \Vert \Lambda^{\frac{\alpha}{2}}u\Vert_{L^2}
  \Vert\Lambda^{\frac{\alpha}{2}}\sigma\Vert_{H^s} + C\Vert \Lambda^{1-\frac{\alpha}{2}}\sigma\Vert_{L^2}\Vert \Lambda^{\frac{\alpha}{2}}u\Vert_{L^2} \Vert \sigma\Vert_{L^\infty}.
\end{aligned}
\end{equation*}
For $K_5$, in light of Lemma \ref{lem:Leibniz}, we obtain
\begin{equation*}
\begin{aligned}
  \vert K_5\vert
  &\le C\Vert \Lambda^{-\frac{\alpha}{2}}\partial_iu^j\Vert_{L^2}
  \Vert\Lambda^{\frac{\alpha}{2}}(u^i\partial_j\Lambda^{-\alpha}\sigma)\Vert_{L^2}\\
  &\le C\Vert \Lambda^{1-\frac{\alpha}{2}}u\Vert_{L^2} \big(\Vert\Lambda^{\frac{\alpha}{2}}u\Vert_{L^\infty}
  \Vert\Lambda^{1-\alpha}\sigma\Vert_{L^2}+\Vert u\Vert_{L^\infty}\Vert\Lambda^{1-\frac{\alpha}{2}}\sigma\Vert_{L^2} \big),
\end{aligned}
\end{equation*}
where we have used the Einstein summation convention on repeated indices.

Note that via the Sobolev embedding and the interpolation inequality,
the following inequalities hold that for every $s\ge s'> \frac{N}{2}+1-\frac{\alpha}{2} $,
\begin{align}
  &\label{eq.D.tos1}\Vert \sigma\Vert_{L^\infty}+\Vert u\Vert_{L^\infty}+\Vert \Lambda^{1-\alpha}\sigma\Vert_{L^2}
  +\Vert \Lambda^{1-\alpha}u\Vert_{L^2}\le C \Vert (\sigma ,u)\Vert_{H^s},\\
  &\label{eq.D.tos2}\Vert\Lambda^{1-\frac{\alpha}{2}}u\Vert_{L^2}+\Vert\Lambda^{\alpha}u\Vert_{L^2}
  +\|\Lambda u\|_{L^2} \le C 
  \Vert \Lambda^{\frac{\alpha}{2}}u\Vert_{H^{s'}}, \\
  & \label{eq.D.tos2.2} \|\Lambda^{2-\frac{3\alpha}{2}} u\|_{L^2} \leq C \|\Lambda^{\frac{\alpha}{2}} u\|_{H^{s'}} ,\quad \textrm{for}\; \alpha\in [\tfrac{2}{3},1],\quad \|\Lambda^{1-\alpha} u\|_{L^2} \leq C \|\Lambda^{\frac{\alpha}{2}} u\|_{H^{s'}},
  \quad \textrm{for}\; \alpha\in(0,\tfrac{2}{3}),
\end{align}
and for every $\alpha\in (0,1)$,
\begin{align}\label{eq.D.tos3}
  \Vert\Lambda^{1-\alpha}\sigma\Vert_{L^\infty} \leq C\Vert \Lambda^{1-\frac{\alpha}{2}}\sigma\Vert_{H^{s'+\alpha-1}},
  \quad \Vert u\Vert_{L^\infty}\le C\Vert \Lambda^{\frac{\alpha}{2}}u\Vert_{H^{s'}}.
\end{align}
Inserting the above estimates on $|I|$, $|J_j|$, $|K_j|$ into \eqref{eq.D.mix1}
and using \eqref{eq.D.tos1}-\eqref{eq.D.tos3}, we get
\begin{equation}\label{eq.D.mix2}
\begin{aligned}
  \frac{\dd}{\dd t}Y(t)
  &+{\mu}\Vert\Lambda^{\frac{\alpha}{2}} u\Vert^{2}_{L^2}
  + \tfrac{\delta_1 \sqrt{\gamma}}{C_1} \Vert\Lambda^{1-\frac{\alpha}{2}}\sigma\Vert_{L^2}^2\\
  & \le C\Vert \sigma,u)\Vert_{H^s} \big(\Vert \Lambda^{1-\frac{\alpha}{2}}\sigma\Vert_{H^{s+\alpha-1}}^2
  + \Vert \Lambda^{\frac{\alpha}{2}}u\Vert_{H^s}^2\big)
  +  \tfrac{\delta_1 \sqrt{\gamma}(1-\alpha)}{2s} \Vert \Lambda^{s+\frac{\alpha}{2}}u\Vert_{L^2}^{2}.
\end{aligned}
\end{equation}
Integrating \eqref{eq.D.mix2} over the time variable and using inequality \eqref{eq.D.Ysim},
we obtain the estimate \eqref{eq.D.Lem1}.
\vskip1mm

\noindent \textbf{Step 2}: $\dot H^s$-estimate.
We shall prove that there exist two positive constants $C=C(\alpha,s,N,\gamma,\mu)$ and
$\widetilde{C}_2 = \widetilde{C}_2(\alpha,s,N)$ for every $s> \frac{N}{2} +1-\frac{\alpha}{2}$,
\begin{align}\label{eq.H.mix3-0}
  &\Vert (\Lambda^{s}\sigma,\Lambda^{s} u)\Vert^{2}_{L^2}
  +\int_0^T \Big(\mu \Vert\Lambda^{s+\frac{\alpha}{2}} u\Vert^{2}_{L^2}
  +\delta_2 \sqrt{\gamma} \Vert\Lambda^{s+\frac{\alpha}{2}}\sigma\Vert_{L^2}^2\Big)\dd t
  \le C\Vert (\sigma_0,u_0)\Vert^{2}_{H^s} + \widetilde{C}_2 \delta_2 \Vert u\Vert_{L^\infty_T(L^2)}^2 \nonumber \\
  &\hspace{.5cm}+ \tfrac{\widetilde{C}_2\delta_2 \mu^2 }{\sqrt{\gamma}}\int_0^T\Vert\Lambda^{\frac{\alpha}{2}}u\Vert_{L^2}^2\dd t
  +C \Vert(\sigma,u)\Vert_{L^\infty(H^s)}\int_0^T \Big(\Vert \Lambda^{\frac{\alpha}{2}}u\Vert_{H^s}^2
  +\Vert \Lambda^{1-\frac{\alpha}{2}}\sigma\Vert_{H^{s+\alpha-1}}^2 \Big)\dd t,
\end{align}
where $0<\delta_2 \leq \min\{1, \frac{\sqrt{\gamma}}{\mu}, \frac{\mu}{\sqrt{\gamma}}\}$ is a constant chosen later.

Multiplying the first and second equation of \eqref{eq.EAS.sigma} with $\Lambda^{2s}\sigma$ and $\Lambda^{2s}u$
respectively, integrating over $\R^N$ and summing them up, we get
\begin{equation}\label{eq.H.sigma.u}
\begin{aligned}
  &\frac{1}{2}\frac{\dd}{\dd t}\big(\Vert \Lambda^{s}\sigma\Vert^{2}_{L^2}
  +\Vert\Lambda^{s} u\Vert^2_{L^2} \big) + \mu\Vert\Lambda^{s+\frac{\alpha}{2}} u\Vert^{2}_{L^2}
  =\tfrac{1-\gamma}{2}\int_{\R^N} \Big(\sigma\Div u\Lambda^{2s}\sigma+\Lambda^{2s}u\cdot(\sigma\nabla\sigma)\Big)\dd x\\
  &-\int_{\R^N} u\cdot\nabla\sigma\Lambda^{2s}\sigma\,\dd x - \int_{\R^N} u\cdot\nabla u\cdot\Lambda^{2s} u\,\dd x
  + \mu \int_{\R^N} \big(u\Lambda^{\alpha}a - \Lambda^\alpha(au)\big)\cdot\Lambda^{2s}u\,\dd x
  \triangleq \sum_{j=1}^{4}\widetilde{I}_j.
\end{aligned}
\end{equation}
In order to develop the smoothing effect of $\sigma$,
we consider the quantity $\int_{\R^N} u\cdot\nabla \Lambda^{2s+\alpha-2}\sigma\dd x$
(this is in accordance with the spectral analysis in high-frequencies).
Arguing as \eqref{eq.D.mix}, we infer that
\begin{align}\label{eq.H.mix}
  &\quad\notag \frac{\dd}{\dd t}\int_{\R^N} u\cdot\nabla \Lambda^{2s+\alpha-2}\sigma\,\dd x
  +\sqrt{\gamma}\Vert\Lambda^{s+\frac{\alpha}{2}}\sigma\Vert_{L^2}^2
  + \mu\int_{\R^N} u\cdot\nabla\Lambda^{2(s+\alpha-1)}\sigma\,
  \dd x- \sqrt{\gamma}\Vert\Lambda^{s+\frac{\alpha}{2}-1}\Div u\Vert_{L^2}^2 \\
  &=\int_{\R^N} u\cdot\nabla\sigma\Lambda^{2s+\alpha-2}\Div u\,\dd x + \tfrac{\gamma-1}{2}
  \int_{\R^N}\sigma \Div u\Lambda^{2s+\alpha-2}\Div u\,\dd x
  -\tfrac{\gamma-1}{4}\int_{\R^N} \sigma^2\Lambda^{2s+\alpha}\sigma\,\dd x \notag \\
  &\quad+\mu \int_{\R^N}\big(u\Lambda^\alpha a-\Lambda^\alpha(a u)\big)\cdot\nabla\Lambda^{2s+\alpha-2}\sigma\,\dd x
  - \int_{\R^N}(u\cdot\nabla u)\cdot\nabla\Lambda^{2s+\alpha-2}\sigma\,\dd x
  \triangleq \sum_{j=1}^5 \widetilde{J}_j.
\end{align}
Through the interpolation inequality and the $L^2$-boundedness of Riesz transform,
there exists a constant $C_2>0$ depending only on $\alpha,s$ such that
\begin{equation}\label{eq.H.L1}
\begin{aligned}
  \left\vert\int_{\R^N}u\cdot\nabla \Lambda^{2s+\alpha-2}\sigma\,\dd x \right\vert
  &\le \Vert \Lambda^{s+\alpha-2}\Div u\Vert_{L^2}\Vert \Lambda^s\sigma\Vert_{L^2}
  \le C\Vert u\Vert_{L^2}^{\frac{1-\alpha}{s}}\Vert \Lambda^s u\Vert_{L^2}^{\frac{s-1+\alpha}{s}}
  \Vert \Lambda^s \sigma\Vert_{L^2}\\
  &\le C_2 \Big(\tfrac{1-\alpha}{2s} \Vert u\Vert_{L^2}^2 + \tfrac{s-1+\alpha}{2s}\Vert \Lambda^s u\Vert_{L^2}^2
  + \tfrac{1}{2}\Vert \Lambda^s\sigma\Vert_{L^2}^2\Big),
\end{aligned}
\end{equation}
and
\begin{equation}\label{eq.H.L2}
\begin{aligned}
  \mu\left\vert\int_{\R^N} u\cdot\nabla\Lambda^{2(s+\alpha-1)}\sigma\,\dd x\right\vert
  &\le \mu\Vert\Lambda^{s+\frac{\alpha}{2}}\sigma\Vert_{L^2}\Vert\Lambda^{s+\frac{3\alpha}{2}-2}\Div u\Vert_{L^2}\\
  &\le  C_2 \tfrac{\mu^2}{\sqrt{\gamma}} \Big(\tfrac{1-\alpha}{2s}\Vert \Lambda^{\frac{\alpha}{2}}u\Vert_{L^2}^2 + \tfrac{s-1+\alpha}{2s}
  \Vert \Lambda^{s+\frac{\alpha}{2}} u\Vert_{L^2}^2 \Big)
  +\tfrac{\sqrt\gamma}{2}\Vert \Lambda^{s+\frac{\alpha}{2}}\sigma\Vert_{L^2}^2,
\end{aligned}
\end{equation}
and
\begin{align}\label{eq.H.L3}
  \sqrt{\gamma}\Vert\Lambda^{s+\frac{\alpha}{2}-1}\Div u\Vert_{L^2}^2
  \le C_2 \tfrac{\sqrt{\gamma}}{\mu} \tfrac{\mu}{2} \Vert\Lambda^{s+\frac{\alpha}{2}}u\Vert_{L^2}^2.
\end{align}
Let $0<\delta_2 \leq \min\{1, \frac{\sqrt{\gamma}}{\mu}, \frac{\mu}{\sqrt{\gamma}}\}$ be a fixed constant chosen later,
then we define
\begin{align*}
  \widetilde{Y}(t)
  \triangleq\Vert \Lambda^s\sigma(t)\Vert^2_{L^2} + \Vert\Lambda^s u(t)\Vert^2_{L^2}
  +  \frac{\delta_2}{C_2} \int_{\R^N} u\cdot \nabla \Lambda^{2s+\alpha-2}\sigma(t)\dd x.
\end{align*}
In view of the inequality \eqref{eq.H.L1}, we see that
\begin{align}\label{eq.H.Ysim}
  \tfrac{1}{2}(\Vert \Lambda^s\sigma\Vert^2_{L^2}
  +\Vert\Lambda^s u\Vert^2_{L^2})-\tfrac{\delta_2 (1-\alpha)}{2s}\Vert u\Vert_{L^2}^2\le
  \widetilde{Y}(t) \le 2(\Vert \Lambda^{s}\sigma\Vert^{2}_{L^2}+ \Vert\Lambda^s u\Vert^2_{L^2})
  +\tfrac{\delta_2 (1-\alpha)}{2s}\Vert u\Vert_{L^2}^2.
\end{align}
Gathering the equations \eqref{eq.H.sigma.u}-\eqref{eq.H.mix} and using \eqref{eq.H.L2}-\eqref{eq.H.L3}, we obtain
\begin{align}\label{eq.H.mix1}
  \frac{1}{2}\frac{\dd}{\dd t} \widetilde{Y}(t)
  +\frac{\mu}{2}\Vert\Lambda^{s+\frac{\alpha}{2}} u \Vert^2_{L^2}
  +\tfrac{\delta_2 \sqrt{\gamma}}{4 C_2}\Vert\Lambda^{s+\frac{\alpha}{2}}\sigma\Vert_{L^2}^2
  \le C \sum_{j=1}^4 |\widetilde{I}_j| + C \sum_{j=1}^5 \vert\widetilde{J}_j\vert
  + \tfrac{\delta_2 \mu^2(1-\alpha)}{4s \sqrt{\gamma}}\Vert \Lambda^{\frac{\alpha}{2}}u\Vert_{L^2}^2.
\end{align}

Next we estimate the terms on the right-hand side of \eqref{eq.H.mix1}.
For every $s>\frac{N}{2}+1-\frac{\alpha}{2}$, we have
\begin{align}
  &\label{eq:embed1}\Vert \nabla f\Vert_{L^\infty}\le C \Vert \Lambda^{\frac{\alpha}{2}} f\Vert_{L^2}^{\frac{2s-2-N+\alpha}{2s'}}
  \Vert \Lambda^{s+\frac{\alpha}{2}} f\Vert_{L^2}^{\frac{2+N-\alpha}{2s}}\le C\Vert\Lambda^{\frac{\alpha}{2}}f\Vert_{H^s},\\
  &\label{eq:embed2}\Vert \nabla f\Vert_{L^\infty}\le C \Vert \Lambda^{1-\frac{\alpha}{2}} f\Vert_{L^2}^{\frac{2s-2-N+\alpha}{2(s-1+\alpha)}}
  \Vert \Lambda^{s+\frac{\alpha}{2}} f\Vert_{L^2}^{\frac{\alpha+N}{2(s-1+\alpha)}}
  \le C\Vert\Lambda^{1-\frac{\alpha}{2}}f\Vert_{H^{s+\alpha-1}}.
\end{align}
By virtue of Lemma \ref{lem:commutator} and inequalities \eqref{eq:embed1}-\eqref{eq:embed2}, we deduce that
\begin{equation*}
\begin{aligned}
  \vert \widetilde{I}_1 \vert
  & = \Big\vert \tfrac{1-\gamma}{2} \int_{\R^N} \Big( \Lambda^s u \cdot \Lambda^s(\sigma\nabla\sigma)
  +\Lambda^s (\sigma \Div u)\, \Lambda^s \sigma \Big) \dd x \Big\vert\\
  & = \Big\vert \tfrac{1-\gamma}{2}\int_{\R^N} \Big(\Lambda^s u\cdot(\Lambda^s(\sigma\nabla\sigma)-\sigma\Lambda^s\nabla\sigma)
  -\Lambda^s u\cdot\nabla\sigma\Lambda^s\sigma+\Lambda^s\sigma(\Lambda^s(\sigma\Div u)-\sigma\Lambda^s\Div u)\Big)\dd x \Big\vert\\
  &\leq C \Big(\Vert \Lambda^s u\Vert_{L^2}\Vert \Lambda^s\sigma\Vert_{L^2}\Vert\nabla\sigma\Vert_{L^\infty}
  +\Vert \Lambda^s\sigma\Vert_{L^2}^2\Vert \Div u\Vert_{L^\infty} \Big)\\
  &\leq C \Vert\sigma\Vert_{H^s} \big(\Vert\Lambda^{1-\frac{\alpha}{2}}\sigma\Vert_{H^{s+\alpha-1}}^2
  + \Vert \Lambda^{\frac{\alpha}{2}}u\Vert_{H^s}^2 \big),
\end{aligned}
\end{equation*}
and
\begin{equation*}
\begin{aligned}
  \vert \widetilde{I}_2\vert
  &= \Big\vert -\int_{\R^N} \big(\Lambda^{s}(u\cdot\nabla\sigma)-u\cdot\nabla\Lambda^{s}\sigma\big)\Lambda^{s}\sigma\dd x
  +\frac{1}{2}\int_{\R^N} \Div u(\Lambda^{s}\sigma)^2 \dd x \Big\vert \\
  &\le C\Vert\nabla\sigma\Vert_{L^\infty}\Vert\Lambda^{s}u\Vert_{L^2}\Vert \Lambda^{s}\sigma\Vert_{L^2}
  +C\Vert \nabla u\Vert_{L^\infty}\Vert\Lambda^{s}\sigma\Vert_{L^2}^2
  +\frac{1}{2}\Vert \Div u\Vert_{L^\infty}\Vert\Lambda^{s}\sigma\Vert_{L^2}^2\\
  &\le C\Vert\sigma\Vert_{H^s} \big(\Vert\Lambda^{1-\frac{\alpha}{2}}\sigma\Vert_{H^{s+\alpha-1}}^2
  + \Vert \Lambda^{\frac{\alpha}{2}}u\Vert_{H^s}^2 \big),
\end{aligned}
\end{equation*}
and
\begin{align*}
  \vert \widetilde{I}_3\vert
  &= \Big\vert  \int_{\R^N} \big(\Lambda^{s}(u\cdot\nabla u)-u\cdot\nabla \Lambda^{s}u\big)\cdot\Lambda^{s}u\dd x
  -\frac{1}{2}\int_{\R^N} \Div u(\Lambda^{s}u)^2\dd x \Big\vert \\
  &\le \Vert\Lambda^{s}(u\cdot\nabla u)-u\cdot\nabla \Lambda^{s}u\Vert_{L^2}\Vert\Lambda^{s}u\Vert_{L^2}+\frac{1}{2}\Vert \Div u\Vert_{L^\infty}\Vert\Lambda^{s}u\Vert_{L^2}^2\\
  &\le C\Vert\nabla u\Vert_{L^\infty}\Vert\Lambda^{s}u\Vert_{L^2}^2 .
\end{align*}
Due to that $H^s\hookrightarrow W^{\alpha,\infty}$ for $\alpha<\frac{2}{3}$, $s>\frac{N}{2}+1-\frac{\alpha}{2}$, $H^{s+\alpha-1}\hookrightarrow W^{3\alpha/2-1,\infty}$ for $\frac{2}{3}\le \alpha\le 1$, we get 
\begin{align}\label{es:Lam-aLinf}
  \|\Lambda^\alpha a\|_{L^\infty}  \leq
  C \|a\|_{H^s} + C \|\Lambda^{1-\frac{\alpha}{2}} a\|_{H^{s+\alpha -1}} 
  \leq C \|\sigma\|_{H^s} + C \|\Lambda^{1-\frac{\alpha}{2}}\sigma\|_{H^{s+\alpha-1}},
\end{align}
and according to Lemmas \ref{lem:Leibniz}, \ref{Lem:composite}, we have
\begin{equation*}
\begin{aligned}
  \vert \widetilde{I}_4 \vert
  &\le C \big(\Vert \Lambda^{s-\frac{\alpha}{2}}(u\Lambda^{\alpha}a)\Vert_{L^2}
  +\Vert\Lambda^{s+\frac{\alpha}{2}}(au)\Vert_{L^2}\big)\Vert \Lambda^{s+\frac{\alpha}{2}}u\Vert_{L^2}\\
  &\le C\big(\Vert \Lambda^{s-\frac{\alpha}{2}}u\Vert_{L^2}\Vert \Lambda^{\alpha}a\Vert_{L^\infty}
  +\Vert u\Vert_{L^\infty}\Vert \Lambda^{s+\frac{\alpha}{2}}a\Vert_{L^2}
  +\Vert a\Vert_{L^\infty}\Vert\Lambda^{s+\frac{\alpha}{2}}u\Vert_{L^2}\big)\Vert \Lambda^{s+\frac{\alpha}{2}}u\Vert_{L^2}\\
  &\le C(\Vert u\Vert_{H^s}+\Vert \sigma\Vert_{H^s})(\Vert \Lambda^{1-\frac{\alpha}{2}}\sigma\Vert_{H^{s+\alpha-1}}^2
  +\Vert \Lambda^{\frac{\alpha}{2}}u\Vert_{H^s}^2).
\end{aligned}
\end{equation*}
For $\widetilde{J}_1$ and $\widetilde{J}_2$ given by \eqref{eq.H.mix}, we use Lemma \ref{lem:Leibniz} and
\eqref{eq:embed1}-\eqref{eq:embed2} to find
\begin{equation*}
\begin{aligned}
  \vert\widetilde{J}_1 \vert
  &\le C\Vert \Lambda^{s+\frac{\alpha}{2}-1}(u\cdot\nabla\sigma)\Vert_{L^2}\Vert\Lambda^{s+\frac{\alpha}{2}-1}\Div u\Vert_{L^2}\\
  &\le C \big(\Vert \Lambda^{s+\frac{\alpha}{2}-1}u\Vert_{L^2}\Vert \nabla\sigma\Vert_{L^\infty}
  +\Vert u\Vert_{L^\infty}\Vert \Lambda^{s+\frac{\alpha}{2}}\sigma\Vert_{L^2} \big)
  \Vert\Lambda^{s+\frac{\alpha}{2}}u\Vert_{L^2} \\
  & \le C \Vert u\Vert_{H^s}(\Vert\Lambda^{1-\frac{\alpha}{2}}\sigma\Vert_{H^{s+\alpha-1}}^2
  + \Vert\Lambda^{s+\frac{\alpha}{2}}u\Vert_{L^2}^2),
\end{aligned}
\end{equation*}
and
\begin{equation*}
\begin{aligned}
  \vert\widetilde{J}_2\vert
  &\le C\Vert\Lambda^{s+\frac{\alpha}{2}-1}(\sigma \Div u)\Vert_{L^2}\Vert\Lambda^{s+\frac{\alpha}{2}-1}\Div u\Vert_{L^2}\\
  &\le C \big( \Vert\Lambda^{s+\frac{\alpha}{2}-1}\sigma \Vert_{L^2}\Vert \Div u\Vert_{L^\infty}
  +\Vert\sigma \Vert_{L^\infty}\Vert\Lambda^{s+\frac{\alpha}{2}}u\Vert_{L^2} \big)
  \Vert\Lambda^{s+\frac{\alpha}{2}}u\Vert_{L^2} \\
  & \le C\Vert\sigma \Vert_{H^{s}}\Vert\Lambda^{\frac{\alpha}{2}}u\Vert_{H^{s}}^2.
\end{aligned}
\end{equation*}
By virtue of Lemma \ref{lem:Leibniz} and the fact that $H^s\hookrightarrow L^\infty$ for every $s>\frac{N}{2}$,
we estimate $\widetilde{J}_3$ as follows:
\begin{align*}
  \vert\widetilde{J}_3\vert
  & \le C\Vert \Lambda^{s+\frac{\alpha}{2}}\sigma^2\Vert_{L^2}\Vert\Lambda^{s+\frac{\alpha}{2}}\sigma\Vert_{L^2} \\
  & \le C\Vert\sigma\Vert_{L^\infty}\Vert \Lambda^{s+\frac{\alpha}{2}}\sigma\Vert_{L^2}^2
  \le C\Vert\sigma \Vert_{H^{s}} \Vert\Lambda^{1-\frac{\alpha}{2}}\sigma\Vert_{H^{s+\alpha-1}}^2.
\end{align*}
Thanks to Lemmas \ref{lem:Leibniz}, \ref{Lem:composite} and inequalities \eqref{eq:embed1}, \eqref{es:Lam-aLinf}, we find
\begin{equation*}
\begin{aligned}
  \vert\widetilde{J}_4\vert
  &\le C\big(\Vert \Lambda^{s-\frac{\alpha}{2}}(u\Lambda^{\alpha}a)\Vert_{L^2}
  +\Vert\Lambda^{s+\frac{\alpha}{2}}(a u)\Vert_{L^2}\big)\Vert \Lambda^{s+\frac{3\alpha}{2}-1}\sigma\Vert_{L^2}\\
  &\le C \big(\Vert \Lambda^{s-\frac{\alpha}{2}}u\Vert_{L^2}\Vert\Lambda^{\alpha}a\Vert_{L^\infty}
  +\Vert u\Vert_{L^\infty}\Vert\Lambda^{s+\frac{\alpha}{2}}a\Vert_{L^2}+\Vert a\Vert_{L^\infty}
  \Vert\Lambda^{s+\frac{\alpha}{2}}u\Vert_{L^2}\big)\Vert \Lambda^{s+\frac{3\alpha}{2}-1}\sigma\Vert_{L^2}\\
  &\le C\big(\Vert u\Vert_{H^s}+\Vert \sigma\Vert_{H^s}\big)
  \big(\Vert \Lambda^{1-\frac{\alpha}{2}}\sigma\Vert_{H^{s+\alpha-1}}^2
  +\Vert \Lambda^{\frac{\alpha}{2}}u\Vert_{H^s}^2\big),
\end{aligned}
\end{equation*}
and
\begin{equation*}
\begin{aligned}
  \vert\widetilde{J}_5\vert
  &\le C\Vert \Lambda^{s+\frac{\alpha}{2}-1}(u\cdot\nabla u)\Vert_{L^2}\Vert \Lambda^{s+\frac{\alpha}{2}}\sigma\Vert_{L^2}\\
  & \le C \big(\Vert \Lambda^{s+\frac{\alpha}{2}-1}u\Vert_{L^2}\Vert \nabla u\Vert_{L^\infty}+\Vert u\Vert_{L^\infty}
  \Vert \Lambda^{s+\frac{\alpha}{2}}u\Vert_{L^2} \big)
  \Vert \Lambda^{s+\frac{\alpha}{2}}\sigma\Vert_{L^2} \\
  &\le C\Vert u\Vert_{H^s} \big(\Vert \Lambda^{1-\frac{\alpha}{2}}\sigma\Vert_{H^{s+\alpha-1}}^2
  +\Vert \Lambda^{\frac{\alpha}{2}}u\Vert_{H^s}^2 \big).
\end{aligned}
\end{equation*}


Inserting the above estimates on $\widetilde{I}_1$-$\widetilde{I}_4$ and $\widetilde{J}_1$-$\widetilde{J}_5$
into \eqref{eq.H.mix1}, we obtain
\begin{equation}\label{eq.H.mix2}
\begin{aligned}
\frac{1}{2}\frac{\dd}{\dd t}\widetilde{Y}(t)
  &+\frac{\mu}{2}\Vert\Lambda^{s+\frac{\alpha}{2}} u(t)\Vert^2_{L^2}
  + \tfrac{\delta_2 \sqrt{\gamma}}{4 C_2} \Vert\Lambda^{s+\frac{\alpha}{2}}\sigma(t) \Vert_{L^2}^2 \\
  &\le C\Vert(\sigma,u)\Vert_{H^s} \big(\Vert \Lambda^{\frac{\alpha}{2}}u\Vert_{H^s}^2
  +\Vert \Lambda^{1-\frac{\alpha}{2}}\sigma\Vert_{H^{s+\alpha-1}}^2 \big)
  +\tfrac{\delta_2 \mu^2(1-\alpha)}{4s \sqrt{\gamma}} \Vert\Lambda^{\frac{\alpha}{2}}u\Vert_{L^2}^2.
\end{aligned}
\end{equation}
By integrating the inequality \eqref{eq.H.mix2} with respect to $t$-variable
and using \eqref{eq.H.Ysim}, we infer that
\begin{align}\label{eq.H.mix3}
  & \Vert (\Lambda^{s}\sigma,\Lambda^{s} u)\Vert^{2}_{L^2}
  + \mu \int_0^T \Big(\Vert\Lambda^{s+\frac{\alpha}{2}} u\Vert^2_{L^2}
  + \tfrac{\delta_2 \sqrt{\gamma}}{C_2} \Vert\Lambda^{s+\frac{\alpha}{2}}\sigma\Vert_{L^2}^2\Big)\dd t
  \le C\Vert (\sigma_0,u_0)\Vert^{2}_{H^s}
  + \tfrac{\delta_2 (1-\alpha)}{s}\Vert u\Vert_{L^\infty_T(L^2)}^2 \nonumber  \\
  &\hspace{.5cm}+ \tfrac{\delta_2 \mu^2(1-\alpha)}{s \sqrt{\gamma}}
  \int_0^T\Vert\Lambda^{\frac{\alpha}{2}}u\Vert_{L^2}^2\dd t
  +C\Vert(\sigma,u)\Vert_{L^\infty_T(H^s)} \int_0^T \big( \Vert \Lambda^{\frac{\alpha}{2}}u\Vert_{H^s}^2
  + \Vert \Lambda^{1-\frac{\alpha}{2}}\sigma\Vert_{H^{s+\alpha-1}}^2 \big) \dd t,
\end{align}
which corresponds to the $\dot H^s$-estimate \eqref{eq.H.mix3-0}, as desired.
\vskip1mm

\noindent \textbf{Step 3}: $H^s$-estimate. By combining \eqref{eq.H.mix3-0} with \eqref{eq.D.Lem1},
and letting $\delta_1,\delta_2>0$ be fixed constants small enough,
 there exists a constant $C = C(\alpha,s,\mu,\gamma,N)>0$ such that for $s>\frac{N}{2}+1 -\frac{\alpha}{2}$,
\begin{equation}\label{eq.Lem2}
\begin{aligned}
  &\Vert (\sigma,u)\Vert^{2}_{L^\infty_T(H^s)}
  +\int_0^T\left(\Vert\Lambda^{\frac{\alpha}{2}} u\Vert^{2}_{H^s}
  +\Vert\Lambda^{1-\frac{\alpha}{2}}\sigma\Vert_{H^{s+\alpha-1}}^2\right)\dd t\\
  &\hspace{2cm}\le C\Vert (\sigma_0,u_0)\Vert^{2}_{H^s}+C\Vert (\sigma,u)\Vert_{L^\infty_T(H^s)}
  \int_0^T\left(\Vert \Lambda^{\frac{\alpha}{2}}u\Vert_{H^s}^2
  +\Vert \Lambda^{1-\frac{\alpha}{2}}\sigma\Vert_{H^{s+\alpha-1}}^2\right)\dd t.
\end{aligned}
\end{equation}
Denote by $X_0 \triangleq \Vert (\sigma_0,u_0)\Vert_{H^s}^2$ and
\begin{align*}
  X(T)\triangleq \Vert (\sigma,u)\Vert^{2}_{L^\infty_T(H^s)}
  +\int_0^T \Big(\Vert\Lambda^{\frac{\alpha}{2}} u\Vert^{2}_{H^s}
  +\Vert\Lambda^{1-\frac{\alpha}{2}}\sigma\Vert_{H^{s+\alpha-1}}^2\Big) \dd t.
\end{align*}
From \eqref{eq.Lem2}, we have
\begin{align*}
  X(T)\le C X_0 + C X^{\frac 3 2}(T).
\end{align*}
Through the continuity argument, we infer that if $X_0 \leq \varepsilon $ with a sufficiently small constant $\varepsilon >0$,
then $X(T)\le 2 C X_0$ for any $T>0$.
\end{proof}

%
%

%
%

\subsection{Proof of Theorem \ref{thm:global}}
The local well-posedness result for the system \eqref{eq.EAS.a} can be established by using the standard arguments, e.g. see \cite{chen2021global,karper2015hydrodynamic, majda2012compressible}.
We present a version of local well-posedness result from \cite[Theorem 2.1]{chen2021global} as follows.
\begin{proposition}\label{prop:local}
Let $s>\frac{N}{2}+1$.
Assume that $(a_0,u_0)\in H^s(\R^N)$ and $\inf\limits_{x\in\R^N}a(x)>-1$.
Then there exist a unique solution $(a,u)$ of the system \eqref{eq.EAS.a} such that
\begin{align}
  (a,u)\in C([0,T],H^s(\R^N)),\quad
  u\in L^2([0,T],H^{s+\frac{\alpha}{2}}(\R^N)),
\end{align}
with some $T>0$.
Moreover, let $T^*$ be the maximal existence time of this solution,
and if $T^*<+\infty$, then there holds:
\begin{align}
  \int_0^{T^*} \Big(\Vert \nabla u(t)\Vert_{L^\infty}+\Vert a(t)\Vert_{W^{1,\infty}}\Big)\dd t=\infty.
\end{align}
\end{proposition}
Under the smallness assumption \eqref{eq:cond0}, or equivalently \eqref{eq.cond0.au},  Proposition \ref{Prop:priori.au} ensures that $(a,u)$ satisfies the $H^s$-estimate \eqref{eq.priori.au}. In particular, we get
\[\|(a, u)\|_{L^\infty_T (H^s)}\leq\|(a_0, u_0)\|_{H^s}<\varepsilon,\]
for every $T<+\infty$. We then apply Sobolev embedding and obtain
\[
\int_0^T \Big(\Vert \nabla u(t)\Vert_{L^\infty}+\Vert a(t)\Vert_{W^{1,\infty}}\Big)\dd t
  \leq C T \|(a, u)\|_{L^\infty_T (H^s)} <CT\varepsilon < +\infty,
\]
with $s>\frac{N}{2}+1$.
In view of Proposition \ref{prop:local}, we must have $T^*=+\infty$, namely the solution exists globally in time. This concludes the proof of Theorem \ref{thm:global}.

\section{Asymptotic behavior}\label{decay}
In this section, we investigate the asymptotic behavior of the solution $(a,u)$ obtained in Theorem \ref{thm:global}. We demonstrate that the solution converges to a steady state:
\[a(t)\to0,\quad\text{and}\quad u(t)\to0,\quad\text{as}\,\,t\to\infty.\]
The density becomes constant due to the effect of pressure, while the velocity converges to a constant due to alignment interactions.

We establish quantitative estimates for this convergence, including explicit decay rates as outlined in Theorem \ref{thm:decay}.

We begin by investigating a linearized system of $(a,v)$ with $v= \Lambda^{-1}\mathrm{div}, u$, focusing on the decay properties of the Green's matrix. 
Subsequently, we establish decay estimates for the solution $(a,u)$ of the system \eqref{eq.EAS.a}. Please refer to Propositions \ref{Prop:upper.rough} and \ref{Prop:upper.refined} for the proof of the estimates in Theorem \ref{thm:decay} (1)-(3).
Finally, to verify the optimality of the decay rate, we derive a lower bound of $\|(a,u)(t)\|_{L^2}$, ensuring that the decay rate precisely coincides with the upper bound. See Proposition \ref{Prop:lower} for the proof of the lower bound estimates in Theorem  \ref{thm:decay} (4).

\subsection{Green's matrix for the linearized system}

By a direct computation (see the appendix for more details), we infer that the solution of the ordinary differential equations \eqref{eq.lineareq4au1} is
\begin{equation}\label{eq.CP}
  \begin{pmatrix}\widehat{a}\\\widehat{v}\end{pmatrix}
  =\widehat{\mathcal{G}}(t,\xi)\begin{pmatrix}\widehat{a_0}\\
  \widehat{v_0}\end{pmatrix}+\int_0^t\widehat{\mathcal{G}}(t-s,\xi)
  \begin{pmatrix}
  \widehat{F}(s) \\
  \widehat{G}(s)
  \end{pmatrix}\dd s,
\end{equation}
where for $\vert\xi\vert^{1-\alpha}\neq{\mu}/({2\sqrt\gamma})$,
\begin{equation}\label{eq:Green1}
  \widehat{\mathcal{G}}(t,\xi) \triangleq
  \begin{pmatrix}
  \widehat{\mathcal{G}}_{11}&\widehat{\mathcal{G}}_{12}\\
  \widehat{\mathcal{G}}_{21}&\widehat{\mathcal{G}}_{22}
  \end{pmatrix}
  = \frac{1}{\lambda_+ - \lambda_{-}}
  \begin{pmatrix}
  \lambda_{+}e^{t\lambda_{-}}-\lambda_{-}e^{t\lambda_{+}}
  & \vert \xi\vert(e^{t\lambda_{-}}-e^{t\lambda_{+}}) \\
  \gamma\vert\xi\vert(e^{t\lambda_{+}}-e^{t\lambda_{-}})
  & \lambda_{+}e^{t\lambda_{+}}-\lambda_{-}e^{t\lambda_{-}}
  \end{pmatrix},
\end{equation}
and for $\vert\xi\vert^{1-\alpha}={\mu}/(2\sqrt\gamma)$,
\begin{equation}\label{eq:Green2}
  \widehat{\mathcal{G}}(t,\xi) \triangleq
  \begin{pmatrix}
  \widehat{\mathcal{G}}_{11}&\widehat{\mathcal{G}}_{12}\\
  \widehat{\mathcal{G}}_{21}&\widehat{\mathcal{G}}_{22}
  \end{pmatrix}
  =e^{-\sqrt\gamma\vert\xi\vert t}
  \begin{pmatrix}
  1+\sqrt\gamma\vert\xi\vert t
  & -\vert\xi\vert t\\
  \gamma\vert\xi\vert t
  &1-\sqrt\gamma\vert\xi\vert t
  \end{pmatrix},
\end{equation}
with $\lambda_\pm$ given by \eqref{def:lamb-pm}.

Now we show the pointwise estimates for the Green matrix $\widehat{\mathcal{G}}(t,\xi)$.
\begin{lemma}\label{lem.point.G.low-high}
There exists constant $C>0$ depending only on $\alpha,N,\mu,\gamma$ such that the following statements hold true.
\begin{enumerate}[(1)]
\item For  $0<\alpha<1$ and the low frequencies $\vert\xi\vert^{1-\alpha}<\mu/(2\sqrt\gamma)$, we have
\begin{align}
  &\vert \widehat{\mathcal{G}}_{11}(t,\xi)\vert \le Ce^{- C t\vert\xi\vert^{2-\alpha}},\quad
  \vert \widehat{\mathcal{G}}_{21}(t,\xi)\vert = \gamma\vert \widehat{\mathcal{G}}_{12}(t, \xi)\vert
  \le C\vert\xi\vert^{1-\alpha}e^{-Ct\vert\xi\vert^{2-\alpha}}, \\
  &\vert \widehat{\mathcal{G}}_{22}(t,\xi) \vert \le C \big(e^{-Ct\vert\xi\vert^\alpha} +
  \vert\xi\vert^{2-2\alpha}e^{-Ct\vert\xi\vert^{2-\alpha}}\big). \label{es:point-G22}
\end{align}
\item For  $0<\alpha<1$ and the high frequencies $\vert\xi\vert^{1-\alpha}\ge\mu/(2\sqrt\gamma)$, we have that for every $i,j=1,2$,
\begin{align}\label{es:point-Gij}
  \vert \widehat{\mathcal{G}}_{ij}(t,\xi)\vert 
  \le Ce^{-Ct} .
\end{align}
\item For $\alpha=1$, we have that for every $i,j=1,2$,
\begin{align}\label{eq:point-Gij2}
  &\vert \widehat{\mathcal{G}}_{ij}(t,\xi)\vert \le Ce^{-C\vert\xi\vert t}.
\end{align}
\end{enumerate}
\end{lemma}

\begin{proof}[Proof of Lemma \ref{lem.point.G.low-high}]
(1) We consider the low-frequency case. By virtue of Taylor's formula,
\begin{align}\label{eq.green21}
  \frac{e^{t\lambda_{+}}-e^{t\lambda_{-}}}{\lambda_{+}-\lambda_{-}}
  =t\int_0^1e^{t(\theta\lambda_{+}+(1-\theta)\lambda_{-})}\dd \theta.
\end{align}
Noticing that $\lambda_+ \le -\frac{\mu}{2}\vert\xi\vert^{\alpha}$ and $\lambda_-\le -\frac{\gamma}{\mu}\vert\xi\vert^{2-\alpha}$
(recalling $\lambda_\pm$ satisfy \eqref{eq:lamb-pm-low}),
we can estimate the above integral as follows:
\begin{align}
  t \int_0^1e^{t(\theta\lambda_{+}+(1-\theta)\lambda_{-})}\dd \theta
  &\le t \int_0^1e^{-t\big(\theta\frac{\mu}{2}\vert\xi\vert^{\alpha}+(1-\theta)\frac{\gamma}{\mu}
  \vert\xi\vert^{2-\alpha}\big)}\dd \theta\notag \\
  & = e^{-\frac{\gamma}{\mu}t\vert\xi\vert^{2-\alpha}} \big(\tfrac{\mu}{2}
  \vert\xi\vert^\alpha - \tfrac{\gamma}{\mu}\vert\xi\vert^{2-\alpha}\big)^{-1}
  \left(1-e^{-t(\frac{\mu}{2}\vert\xi\vert^{\alpha}-\frac{\gamma}{\mu}\vert\xi\vert^{2-\alpha})}\right) \nonumber \\
  & \le \frac{4}{\mu} \vert\xi\vert^{-\alpha}e^{-\frac{\gamma}{\mu}t\vert\xi\vert^{2-\alpha}},\notag
\end{align}
where in the last line we have used the fact that $\big(1- \frac{2\gamma}{\mu^2} |\xi|^{2-2\alpha}\big)^{-1} \leq 2$.
From the definitions of $\widehat{\mathcal{G}}_{12}$ and $\widehat{\mathcal{G}}_{21}$ in \eqref{eq:Green1}, we have
\begin{align}\notag
  \vert \widehat{\mathcal{G}}_{21}(t,\xi)\vert = \gamma\vert \widehat{\mathcal{G}}_{12}(t,\xi)\vert
  \le C\vert\xi\vert^{1-\alpha} e^{-\frac{\gamma}{\mu}t\vert\xi\vert^{2-\alpha}}.
\end{align}
We rewrite $\widehat{\mathcal{G}}_{11}$ and $\widehat{\mathcal{G}}_{22}$ as follows:
\begin{align}\label{eq.green11}
  \widehat{\mathcal{G}}_{11}(t,\xi)
  =e^{t\lambda_{-}} -\frac{\lambda_{-}(e^{t\lambda_{+}}
  -e^{t\lambda_{-}})}{\lambda_{+}-\lambda_{-}}, \quad
  \widehat{\mathcal{G}}_{22}(t,\xi)
  =e^{t\lambda_{+}} +\frac{\lambda_{-}(e^{t\lambda_{+}}-e^{t\lambda_{-}})}{\lambda_{+}-\lambda_{-}}.
\end{align}
Hence, the triangle inequality gives
\begin{align}\notag
  \vert \widehat{\mathcal{G}}_{11}(t,\xi)\vert
  \le e^{-\frac{\gamma}{\mu}t\vert\xi\vert^{2-\alpha}}+C\vert\xi\vert^{2-2\alpha}e^{-Ct\vert\xi\vert^{2-\alpha}}
  \le Ce^{-Ct\vert\xi\vert^{2-\alpha}}.
\end{align}
Similarly, we can obtain \eqref{es:point-G22} concerning the estimate of $\vert \widehat{\mathcal{G}}_{22}\vert$.
\vskip1mm

\noindent(2) We prove \eqref{es:point-Gij} by respectively considering the following three high-frequency cases. 
\begin{enumerate}[$\triangleright$]
\item If $\vert\xi\vert^{1-\alpha}=\mu/(2\sqrt\gamma)$, the estimate \eqref{es:point-Gij} directly follows from \eqref{eq:Green2}.

\item  For $\mu/(2\sqrt\gamma)<\vert\xi\vert^{1-\alpha}\le\mu/\sqrt\gamma$,
recalling $\lambda_\pm$ are given by \eqref{eq:lamb-pm-high},
we see that
\begin{align}\notag
  \theta\lambda_{+}+(1-\theta)\lambda_{-}
  =\frac{-\mu\vert\xi\vert^{\alpha}}{2} \Big{(}1-i(1-2\theta)
  \sqrt{\tfrac{4\gamma}{\mu^2}\vert \xi\vert^{2-2\alpha}-1} \Big{)},
\end{align}
and
\begin{align}\notag
  \Big| \frac{e^{t\lambda_+} - e^{t\lambda_-}}{\lambda_+ -\lambda_-} \Big| \leq
  t\int_0^1\big\vert e^{t(\theta\lambda_{+}+(1-\theta)\lambda_{-})}\big\vert \dd \theta
  = t e^{-\frac{\mu}{2}t\vert\xi\vert^{\alpha}},
\end{align}
and thus
\begin{align}\notag
  \vert \widehat{\mathcal{G}}_{21}(t,\xi)\vert = \gamma\vert \widehat{\mathcal{G}}_{12}(t,\xi)\vert
  \le Ct\vert \xi\vert e^{-\frac{\mu}{2}t\vert\xi\vert^{\alpha}}
  \le Ce^{-Ct\vert\xi\vert^{\alpha}}
  \le Ce^{-Ct}.
\end{align}
Noting that $\Re(\lambda_{\pm}) = -\frac{\mu\vert\xi\vert^\alpha}{2}$
and $\vert\lambda_- \vert=\sqrt\gamma\vert \xi\vert$,
and in view of the equality \eqref{eq.green11}, we have 
\begin{align}\notag
  \vert \widehat{\mathcal{G}}_{11}(t,\xi) \vert + |\widehat{\mathcal{G}}_{22}(t,\xi)|
  \le  Ce^{-\frac{\mu}{2}t\vert\xi\vert^{\alpha}}+ Ct\vert \xi\vert e^{-\frac{\mu}{2}t\vert\xi\vert^{\alpha}}
  \le Ce^{-Ct}.
\end{align}

\item For $\vert\xi\vert^{1-\alpha}>\mu/\sqrt\gamma$, noting that
\begin{align}\notag
  \vert\lambda_{+}-\lambda_{-}\vert=\mu\vert\xi\vert^{\alpha}\sqrt{\tfrac{4\gamma}{\mu^2}\vert \xi\vert^{2-2\alpha}-1}
  \ge \mu\vert\xi\vert^{\alpha}\sqrt{\tfrac{3\gamma}{\mu^2}\vert \xi\vert^{2-2\alpha}}=\sqrt{3\gamma}\vert\xi\vert,
\end{align}
and recalling that $\Re(\lambda_{\pm})=-\frac{\mu\vert\xi\vert^{\alpha}}{2}$ and
$\vert\lambda_{\pm}\vert=\sqrt\gamma\vert \xi\vert$, we have
\begin{align}\notag
  \vert \widehat{\mathcal{G}}_{11}(t,\xi)\vert
  \le \frac{\vert\lambda_{+}\vert}{\vert\lambda_{+}-\lambda_{-}\vert}\vert e^{t\lambda_-}\vert
  +\frac{\vert\lambda_{-}\vert}{\vert\lambda_{+}-\lambda_{-}\vert}\vert e^{t\lambda_+}\vert
  \le Ce^{-Ct\vert\xi\vert^{\alpha}}
  \le Ce^{-Ct}.
\end{align}
Similarly, we can obtain the estimates of $\vert \widehat{\mathcal{G}}_{12}\vert$, $\vert \widehat{\mathcal{G}}_{22}\vert$
and $\vert \widehat{\mathcal{G}}_{22}\vert$ as in \eqref{es:point-Gij}.
\vskip1mm
\end{enumerate}

\noindent (3) We prove \eqref{eq:point-Gij2} by analyzing the following three cases.
\begin{enumerate}[$\triangleright$]
\item If $4\gamma<\mu^2$, we see that $\lambda_{\pm}={-\vert\xi\vert}{(}\mu\pm\sqrt{\mu^2-{4\gamma}}{)}/2$
and $\lambda_+ -\lambda_- = - \sqrt{\mu^2 -4\gamma} |\xi|$,
and thus
\begin{align}\notag
  \vert e^{t\lambda_{\pm}}\vert\le Ce^{-C\vert\xi\vert t},\quad
  \left\vert\frac{\lambda_{\pm}}{\lambda_{+}-\lambda_{-}}\right\vert\le C,\quad\text{and }
  \frac{\vert\xi\vert}{|\lambda_{+}-\lambda_{-}|}\le C.
\end{align}
Hence, the estimate \eqref{eq:point-Gij2} follows from the above inequalities
and the formula of $\vert \widehat{\mathcal{G}}_{ij}\vert$ in \eqref{eq:Green1}.

\item If $4\gamma=\mu^2$, in view of \eqref{eq:Green2} and the fact that
$\vert\xi\vert t e^{-\sqrt\gamma\vert\xi\vert t}\le Ce^{-\frac{\sqrt\gamma}{2}\vert\xi\vert t}$,
we directly obtain \eqref{eq:point-Gij2}.

\item If $4\gamma>\mu^2$, we see that $\lambda_{\pm}
={-\vert\xi\vert}{(}\mu\pm i\sqrt{4\gamma-\mu^2}{)}/2$,
$\lambda_+ -\lambda_- = -i \sqrt{4\gamma -\mu^2}|\xi|$,
and thus
\begin{align}\notag
  \vert e^{t\lambda_{\pm}}\vert\le Ce^{-\frac{\mu}{2}\vert\xi\vert t},\quad
  \left\vert\frac{\lambda_{\pm}}{\lambda_{+}-\lambda_{-}}\right\vert\le C,\quad\text{and }
  \frac{\vert\xi\vert}{|\lambda_{+}-\lambda_{-}|}\le C,
\end{align}
which in combination with \eqref{eq:Green1} leads to \eqref{eq:point-Gij2}.
\end{enumerate}
Hence, we complete the proof of Lemma \ref{lem.point.G.low-high}.
\end{proof}

Next, based on the above pointwise estimates,
we provide the following two lemmas concerning the exact decay estimates related to the Green matrix
$\widehat{\mathcal{G}}(t,\xi)$.
\begin{lemma}\label{lem:GM.L2}
Let $0<\alpha\le1$, $s\ge 0$ and $f\in \dot{H}^s\cap L^1(\R^N)$.
Then we have
\begin{align}
  &\label{eq.G11.Decay}\Vert \vert\xi\vert^s \widehat{\mathcal{G}}_{11}(t)\widehat{f}\Vert_{L^2}
  \le C \langle t\rangle^{-\frac{N+ 2s}{2(2-\alpha)}}\Vert f\Vert_{L^1\cap\dot{H}^{s}},\\
  &\label{eq.G12.Decay}\Vert \vert\xi\vert^{s}\widehat{\mathcal{G}}_{21}(t)\widehat{f}\Vert_{L^2}
  =\gamma\Vert \vert\xi\vert^{s}\widehat{\mathcal{G}}_{12}(t)\widehat{f}\Vert_{L^2}
  \le C\langle t\rangle^{-\frac{N + 2s +2-2\alpha}{2(2-\alpha)}}\Vert f\Vert_{ L^1\cap\dot{H}^{s}},\\
  &\label{eq.G22.Decay}\Vert \vert\xi\vert^{s}\widehat{\mathcal{G}}_{22}(t)\widehat{f}\Vert_{L^2}
  \le C\langle t\rangle^{-\min\left\{\frac{N+ 2s}{2\alpha},\frac{N+ 2s +4-4\alpha}{2(2-\alpha)}\right\}}
  \Vert f\Vert_{ L^1\cap\dot{H}^s},
\end{align}
where $C>0$ depends only on $\alpha,s,N,\mu,\gamma$.
\end{lemma}

\begin{proof}[Proof of Lemma \ref{lem:GM.L2}]
For $\alpha=1$, the estimates \eqref{eq.G11.Decay}-\eqref{eq.G22.Decay}
naturally follow from \eqref{eq:point-Gij2} and Lemma \ref{lem:heat.decay}.
Now we consider the case $0<\alpha<1$.
Denote by 
\begin{equation}\label{eq:D}
D \triangleq \{\xi\in\R^N:\vert\xi\vert^{1-\alpha}<\mu/(2\sqrt\gamma)\}.
\end{equation}
Thanks to Lemma \ref{lem.point.G.low-high}, we have for $t\le 1$,
\begin{equation}\label{eq.G11.D1}
\begin{aligned}
  \Vert\vert\xi\vert^{s}\widehat{\mathcal{G}}_{11}(t)\widehat{f}\Vert_{L^2}
  &\le \Vert\vert\xi\vert^{s}\widehat{\mathcal{G}}_{11}(t)\widehat{f}\Vert_{L^2(D)}
  +\Vert\vert\xi\vert^{s}\widehat{\mathcal{G}}_{11}(t)\widehat{f}\Vert_{L^2(D^c)}\\
  &\le C\Vert e^{-C\vert\xi\vert^{2-\alpha}t}\Vert_{L^\infty}\Vert
  \vert\xi\vert^s \widehat{f}\Vert_{L^2}+e^{-C t}\Vert\vert\xi\vert^{s}\widehat{f}\Vert_{L^2}
  \le C \Vert {f}\Vert_{\dot{H}^s},
\end{aligned}
\end{equation}
and for $t>1$,
\begin{equation}\label{eq.G11.D2}
\begin{aligned}
  \Vert\vert\xi\vert^{s}\widehat{\mathcal{G}}_{11}(t)\widehat{f}\Vert_{L^2}
  &\le \Vert\vert\xi\vert^{s}\widehat{\mathcal{G}}_{11}(t)\widehat{f}\Vert_{L^2(D)}
  +\Vert\vert\xi\vert^{s}\widehat{\mathcal{G}}_{11}(t)\widehat{f}\Vert_{L^2(D^c)}\\
  &\le C\Vert \vert\xi\vert^{s}e^{-C\vert\xi\vert^{2-\alpha}t}\Vert_{L^2}\Vert \widehat{f}\Vert_{L^\infty}
  +e^{-C t}\Vert\vert\xi\vert^{s}\widehat{f}\Vert_{L^2}
  \le C t^{-\frac{N+2s}{2(2-\alpha)}}\Vert f\Vert_{L^1\cap\dot{H}^s}.
\end{aligned}
\end{equation}
Combining \eqref{eq.G11.D1} and \eqref{eq.G11.D2} leads to \eqref{eq.G11.Decay}.
In a similar way, we can analogously prove the estimates \eqref{eq.G12.Decay} and \eqref{eq.G22.Decay}.
\end{proof}

\begin{lemma}\label{lem:GM.L1}
Let $0<\alpha\le1$ and $f,\widehat{f}\in L^1(\R^N)$.
Then we have
\begin{align}
  &\label{eq.G11.Decay.Linf}
  \Vert \widehat{\mathcal{G}}_{11}(t)\widehat{f}\Vert_{L^1}
  \le C \langle t\rangle^{-\frac{N}{2-\alpha}} \big(\Vert f\Vert_{L^1}+\Vert \widehat{f}\Vert_{L^1} \big),\\
  &\label{eq.G12.Decay.Linf}
  \Vert \widehat{\mathcal{G}}_{21}(t)\widehat{f}\Vert_{L^1}
  =\gamma\Vert \widehat{\mathcal{G}}_{12}(t)\widehat{f}\Vert_{L^2}
  \le C\langle t\rangle^{-\frac{N+1-\alpha}{2-\alpha}}
  \big(\Vert f\Vert_{L^1}  + \Vert \widehat{f}\Vert_{L^1}\big),\\
  &\label{eq.G22.Decay.Linf}
  \Vert \widehat{\mathcal{G}}_{22}(t)\widehat{f}\Vert_{L^1}
  \le C\langle t\rangle^{-\min\left\{\frac{N}{\alpha},\frac{N+2-2\alpha}{2-\alpha}\right\}}
  \big(\Vert f\Vert_{L^1}  + \Vert \widehat{f}\Vert_{L^1}\big),
\end{align}
with $C>0$ depending only on $\alpha,s,N,\mu,\gamma$.
\end{lemma}

\begin{proof}[Proof of Lemma \ref{lem:GM.L1}]
For $\alpha=1$, \eqref{eq.G11.Decay.Linf}-\eqref{eq.G22.Decay.Linf} can be deduced from \eqref{eq:point-Gij2}
and Lemma \ref{lem:heat.decay.L1}.
Next we treat the case $0<\alpha <1$. Recalling $D$ is defined in \eqref{eq:D},
and in light of Lemma \ref{lem.point.G.low-high}, we have that for $t\le 1$,
\begin{equation}\label{eq.G11.D1.Linf}
\begin{aligned}
  \Vert\widehat{\mathcal{G}}_{11}(t)\widehat{f}\Vert_{L^1}
  &\le \Vert\widehat{\mathcal{G}}_{11}(t)\widehat{f}\Vert_{L^1(D)}
  +\Vert\widehat{\mathcal{G}}_{11}(t)\widehat{f}\Vert_{L^1(D^c)}\\
  &\le C\Vert e^{-C\vert\xi\vert^{2-\alpha}t}\Vert_{L^\infty}\Vert \widehat{f}\Vert_{L^1}
  +e^{-C t}\Vert\widehat{f}\Vert_{L^1}
  \le C \Vert \widehat{f}\Vert_{L^1},
\end{aligned}
\end{equation}
and for $t>1$,
\begin{equation}\label{eq.G11.D2.Linf}
\begin{aligned}
  \Vert\widehat{\mathcal{G}}_{11}(t)\widehat{f}\Vert_{L^1}
  &\le \Vert\widehat{\mathcal{G}}_{11}(t)\widehat{f}\Vert_{L^1(D)}
  +\Vert\widehat{\mathcal{G}}_{11}(t)\widehat{f}\Vert_{L^1(D^c)}\\
  &\le C\Vert e^{-C\vert\xi \vert^{2-\alpha}t}\Vert_{L^1}\Vert \widehat{f}\Vert_{L^\infty}
  +e^{-C t}\Vert\widehat{f}\Vert_{L^1}
  \le C t^{-\frac{N}{2-\alpha}}(\Vert f\Vert_{L^1}+\Vert \widehat{f}\Vert_{L^1}).
\end{aligned}
\end{equation}
Combining these two estimates leads to \eqref{eq.G11.Decay.Linf}.
Similarly, we can analogously show the estimates \eqref{eq.G12.Decay.Linf} and \eqref{eq.G22.Decay.Linf}.
\end{proof}

\subsection{Rough upper bounds}
In this subsection, we shall establish a rough upper bound for the decay rate of the global smooth solution $(a,u)$ to the system \ref{eq.EAS.a}.

Before proceeding forward we introduce some notations.
Denote by
\begin{align}
  & U_\beta(t) \triangleq \sup_{0\le \tau\le t}\langle \tau\rangle^{\frac{N+2\beta}{2(2-\alpha)}}
  \Vert (\Lambda^\beta a,\Lambda^\beta u)(\tau)\Vert_{L^2}, \quad \\
  & V(t) \triangleq \sup_{0\le \tau\le t}\langle \tau\rangle^{\frac{N}{2-\alpha}}
  \Vert (\widehat{a},\widehat{u})(\tau)\Vert_{L^1},\\
  &{E}(t) \triangleq \sup_{0\le \tau\le t}\Vert (a,u)(\tau)\Vert_{H^s}.
\end{align}
Our main result of this subsection is as follows.
\begin{proposition}\label{Prop:upper.rough}
Let $\alpha\in(0,1]$, $N\ge 2$ and $s > \frac{N}{2}+1$.
Suppose that $(a_0,u_0) \in H^s \cap L^1(\R^N)$ and $(a,u)$ is a global solution of system \eqref{eq.EAS.a} such that
\begin{align*}
  \sup_{0\le t<+\infty} \big(\Vert a(t)\Vert_{H^s}+\Vert u(t)\Vert_{H^s} \big) < +\infty.
\end{align*}
Then we have
\begin{align}
  &\Vert(a,u)(t)\Vert_{L^2}\le C\langle t\rangle^{-\frac{N}{2(2-\alpha)}},  \label{eq.decay1}\\
  & \Vert (a,u)(t)\Vert_{L^\infty}\le C\langle t\rangle^{-\frac{N}{2-\alpha}}, \label{eq.decay2} \\
  & \Vert (\nabla a,\nabla u)(t)\Vert_{L^2}\le C\langle t\rangle^{-\frac{N+2}{2(2-\alpha)}}. \label{eq.decay3}
\end{align}
\end{proposition}

\begin{proof}[Proof of Proposition \ref{Prop:upper.rough}]
We divide the proof into four steps.

\noindent \textbf{Step 1}: Estimation of $U_0(t)$. We intend to prove that
\begin{align}\label{eq.U0}
  U_0(t)\le C \Vert (a_0,u_0)\Vert_{ L^1\cap L^2} + C\int_0^t W_1(t,\tau)(U_0+U_1)(\tau)E(\tau)\dd \tau,
\end{align}
where the weight function $W_1$ is defined by
\begin{align*}
  W_1(t,\tau)& \triangleq \langle t-\tau\rangle^{-\frac{N+2}{2(2-\alpha)}}
  \langle \tau\rangle^{-\frac{N}{2(2-\alpha)}}
  \langle t\rangle^{\frac{N}{2(2-\alpha)}}
  \\&\quad
  +\langle t-\tau\rangle^{-\min \big\{\frac{N}{2\alpha},\frac{N+2(1-\alpha)}{2(2-\alpha)}\big\}}
  \langle \tau \rangle^{- \frac{N+2\alpha}{2(2-\alpha)}}
  \langle t\rangle^{\frac{N}{2(2-\alpha)}}.
\end{align*}

We first consider the incompressible part $\mathbb{P}u$ with $\mathbb{P} \triangleq \Id - \nabla \Delta^{-1} \Div$.
Recalling that $\mathbb{P}u$ satisfies the third equation of \eqref{eq.lineareq4au}, we deduce that
\begin{align}\label{eq.Pu.decay}
  \mathbb{P}u(t) =e^{-\mu t\Lambda^{\alpha}}\mathbb{P}u_0
  +\int_0^te^{-\mu (t-\tau)\Lambda^{\alpha}}\mathbb{P}\widetilde{H}(\tau)\dd \tau,
\end{align}
with
\begin{align}\label{def:tild-H}
  \widetilde{H} \triangleq \mu \big( u\Lambda^{\alpha}a-\Lambda^{\alpha}(a u)\big)- u\cdot\nabla u.
\end{align}
By virtue of Lemma \ref{lem:heat.decay}, the first term on the right-hand side of \eqref{eq.Pu.decay} can be estimated as
\begin{align}\label{eq.Pu1.decay}
  \Vert e^{-\mu t\Lambda^{\alpha}}\mathbb{P}u_0\Vert_{L^2}
  \le C\Vert e^{-\mu t\vert\xi\vert^{\alpha}}u_0\Vert_{L^2}
  \le C\langle t\rangle^{-\frac{N}{2\alpha}}\Vert u_0\Vert_{L^1\cap L^2}.
\end{align}
Thanks to Lemma \ref{lem:commutator}, we have
\begin{align}\label{eq.tildeH.L1}
  \Vert \widetilde{H} \Vert_{L^1 \cap L^2}
  \leq C\big(\Vert \Lambda^\alpha u\Vert_{L^2} \Vert a\Vert_{L^2\cap L^\infty}
  + \Vert u\Vert_{L^2\cap L^\infty} \Vert\nabla u\Vert_{L^2}\big).
\end{align}
Using Lemma \ref{lem:heat.decay} and the above inequality, we find
\begin{align}\label{eq.Pu2.decay}
  & \int_0^t\Vert e^{-\mu (t-\tau)\Lambda^{\alpha}}\mathbb{P} \widetilde{H}(\tau)\Vert_{L^2}\dd \tau
  \le C\int_0^t\langle t-\tau\rangle^{-\frac{N}{2\alpha}} \Vert \widetilde{H}(\tau)\Vert_{L^1 \cap L^2} \dd \tau\notag\\
  & \leq C \int_0^t\langle t-\tau\rangle^{-\frac{N}{2\alpha}} \Big( \|\Lambda^\alpha u\|_{L^2} \|a\|_{L^2\cap L^\infty} + \|\nabla u\|_{L^2} \|u\|_{L^2\cap L^\infty} \Big)\dd\tau \nonumber \\
  &\le C\int_0^t\langle t-\tau\rangle^{-\frac{N}{2\alpha}}
  \langle \tau \rangle^{- \frac{N+2\alpha}{2(2-\alpha)}}  (EU_\alpha+EU_1)(\tau)   \dd \tau.
\end{align}
Collecting \eqref{eq.Pu.decay}, \eqref{eq.Pu1.decay} and \eqref{eq.Pu2.decay} yields
\begin{align}\label{eq.Pu.L2.decay1}
  \Vert\mathbb{P}u\Vert_{L^2}
  &\le C\langle t\rangle^{-\frac{N}{2(2-\alpha)}}\Big(\Vert u_0\Vert_{L^1\cap L^2}
  + \int_0^tW_1(t,\tau)\big(U_\alpha+U_{1}\big)( \tau)E(\tau)\dd \tau\Big).
\end{align}

Now we consider the compressible part $(a,v)$.
From the equality \eqref{eq.CP}, we have
\begin{align}\label{eq.sigma.L2.decay}
  \Vert a(t)\Vert_{L^2}
  \le \Vert\widehat{\mathcal{G}}_{11}(t)\widehat{a_0}\Vert_{L^2}+\Vert\widehat{\mathcal{G}}_{12}(t)\widehat{v_0}\Vert_{L^2}
  +\int_0^t \Big(\Vert\widehat{\mathcal{G}}_{11}(t-\tau)\widehat{F}\Vert_{L^2}
  +\Vert\widehat{\mathcal{G}}_{12}(t-\tau)\widehat{G} \Vert_{L^2} \Big)\dd \tau,
\end{align}
and
\begin{align}\label{eq.v.L2.decay}
  \Vert v(t)\Vert_{L^2}
  \le \Vert\widehat{\mathcal{G}}_{21}(t)\widehat{a_0}\Vert_{L^2}
  +\Vert\widehat{\mathcal{G}}_{22}(t)\widehat{v_0}\Vert_{L^2}
  +\int_0^t \Big(\Vert\widehat{\mathcal{G}}_{21}(t-\tau)\widehat{F}\Vert_{L^2}
  +\Vert\widehat{\mathcal{G}}_{22}(t-\tau)\widehat{G}\Vert_{L^2}\Big) \dd \tau,
\end{align}
where $F =\Div(au)$ and $G$ is defined by \eqref{def:G}.
In view of Lemma \ref{lem:GM.L2}, we get
\begin{align}\label{eq.G11.int.decay}
  \Vert\widehat{\mathcal{G}}_{11}(t)\widehat{a_0}\Vert_{L^2}
  \le C \langle t\rangle^{-\frac{N}{2(2-\alpha)}}\Vert a_0\Vert_{L^1\cap L^2},
\end{align}
\begin{align}\label{eq.G12.int.decay}
  \Vert\widehat{\mathcal{G}}_{12}(t)\widehat{v_0}\Vert_{L^2}
  + \Vert\widehat{\mathcal{G}}_{21}(t) \widehat{a_0} \Vert_{L^2}
  \le C\Vert\widehat{\mathcal{G}}_{12}(t) (\widehat{u_0}, \widehat{a_0})\Vert_{L^2}
  \le C \langle t\rangle^{-\frac{N+2(1-\alpha)}{2(2-\alpha)}}\Vert (u_0,a_0)\Vert_{L^1\cap L^2},
\end{align}
\begin{align}\label{eq.G22.int.decay}
  \Vert\widehat{\mathcal{G}}_{22}(t)\widehat{v_0}\Vert_{L^2}
  \le C\Vert\widehat{\mathcal{G}}_{22}(t)\widehat{u_0}\Vert_{L^2}
  \le C \langle t\rangle^{-\frac{N+4(1-\alpha)}{2(2-\alpha)}}\Vert u_0\Vert_{L^1\cap L^2}.
\end{align}
Taking advantage of Lemma \ref{lem:GM.L2} and the following inequality
\begin{align*}
  \Vert au\Vert_{\dot{H}^1}
  \le C \big(\Vert u\Vert_{L^\infty}\Vert a\Vert_{\dot{H}^1}+\Vert a\Vert_{L^\infty}\Vert u\Vert_{\dot{H}^1}\big)
  \leq C \|(a,u)\|_{\dot H^1} \|({a},{u})\|_{L^\infty},
\end{align*}
we infer that
\begin{align}\label{eq.G11.F.decay}
  \int_0^t\Vert\widehat{\mathcal{G}}_{11}(t-\tau)\widehat{F}(\tau)\Vert_{L^2}\dd \tau
  &\le C\int_0^t\langle t-\tau\rangle^{-\frac{N+2}{2(2-\alpha)}}\Vert (au)(\tau)
  \Vert_{L^1\cap\dot{H}^{1}}\dd \tau\notag\\
  & \leq C \int_0^t \langle t-\tau \rangle^{-\frac{N+2}{2(2-\alpha)}} \big(
  \|a\|_{L^2} \|u\|_{L^2} +  \|(a,u)\|_{\dot H^1} \|({a},{u})\|_{L^\infty}  \big)
  \dd \tau \nonumber \\
  &\le C\int_0^t\langle t-\tau\rangle^{-\frac{N+2}{2(2-\alpha)}}
  \langle \tau\rangle^{-\frac{N}{2(2-\alpha)}}\big(EU_0 + E U_1\big)(\tau)\dd \tau,
\end{align}
and
\begin{align}\label{eq.G21.F.decay}
  \int_0^t\Vert\widehat{\mathcal{G}}_{21}(t-\tau)\widehat{F}(\tau) \Vert_{L^2}\dd \tau
  &\le C\int_0^t\langle t-\tau\rangle^{-\frac{N+2(2-\alpha)}{2(2-\alpha)}}
  \Vert (au)(\tau)\Vert_{L^1\cap\dot{H}^1}\dd \tau\notag\\
  &\le C\int_0^t\langle t-\tau\rangle^{-\frac{N+2}{2(2-\alpha)}}
  \langle \tau\rangle^{-\frac{N}{2(2-\alpha)}}\big(EU_0 + E U_1\big)(\tau)\dd \tau.
\end{align}
Denote by
\begin{align}\label{def:tild-G}
  \widetilde{G}\triangleq \mu \big(u\Lambda^{\alpha}a-\Lambda^{\alpha}(au)\big)-u\cdot\nabla u
  -\gamma\big((1+a)^{\gamma-2}-1\big)\nabla a,
\end{align}
then Lemmas \ref{lem:commutator} and \ref{Lem:composite} yield that
\begin{align}\label{eq.tildeG.L1.Decay}
  \Vert \widetilde{G}\Vert_{L^1\cap L^2}
  \le C \big(\Vert \Lambda^\alpha u\Vert_{L^2} \Vert a\Vert_{L^2\cap L^\infty}
  + \Vert u\Vert_{L^2\cap L^\infty} \Vert\nabla u\Vert_{L^2}
  + \Vert a\Vert_{L^2\cap L^\infty}\Vert\nabla a\Vert_{L^2}\big).
\end{align}
Arguing as \eqref{eq.G11.F.decay} we find
\begin{align}\label{eq.G12.G.decay}
  &\int_0^t\Vert \widehat{\mathcal{G}}_{12}(t-\tau)\widehat{G}(\tau)\Vert_{L^2}\dd \tau
  \le C\int_0^t\langle t-\tau\rangle^{-\frac{N+2(1-\alpha)}{2(2-\alpha)}}
  \Vert \widetilde{G}(\tau)\Vert_{L^1\cap L^2} \dd \tau\notag\\
  & \leq C\int_0^t\langle t-\tau\rangle^{-\frac{N+2(1-\alpha)}{2(2-\alpha)}}
  \Big(\Vert \Lambda^\alpha u\Vert_{L^2} \Vert a\Vert_{L^2\cap L^\infty}
  + \Vert(\nabla a,\nabla u)\Vert_{L^2}\Vert u\Vert_{L^2\cap L^\infty} \Big) \dd \tau\nonumber \\
  &\le C\int_0^t\langle t-\tau\rangle^{-\frac{N+2(1-\alpha)}{2(2-\alpha)}}
  \langle \tau\rangle^{-\frac{N+2\alpha}{2(2-\alpha)}}\big(EU_1+ EU_\alpha\big)(\tau)\dd \tau ,
\end{align}
and
\begin{align}\label{eq.G22.G.decay}
  \int_0^t\Vert\widehat{\mathcal{G}}_{22}(t-\tau)\widehat{G}(\tau)\Vert_{L^2}\dd \tau
  &\le C\int_0^t\langle t-\tau\rangle^{-\frac{N+4-4\alpha}{2(2-\alpha)}}
  \Vert \widetilde{G}(\tau)\Vert_{ L^1\cap L^2}\dd \tau\notag\\
  &\le C\int_0^t\langle t-\tau\rangle^{-\frac{N+2(1-\alpha)}{2(2-\alpha)}}
  \langle \tau\rangle^{-\frac{N+2\alpha}{2(2-\alpha)}}\big( EU_1+EU_\alpha\big)(\tau)\dd \tau .
\end{align}
Inserting \eqref{eq.G11.int.decay}-\eqref{eq.G21.F.decay}
and \eqref{eq.G12.G.decay}-\eqref{eq.G22.G.decay} into \eqref{eq.sigma.L2.decay}-\eqref{eq.v.L2.decay}, we have
\begin{align}\label{eq.sigma.L2.decay1}
  \Vert (a,v)(t)\Vert_{L^2}
  \le C\langle t\rangle^{-\frac{N}{2(2-\alpha)}}
  \Big(\Vert (a_0,u_0)\Vert_{L^1\cap L^2} + \int_0^tW_1(t,\tau)\big(U_0 + U_\alpha + U_1\big)(\tau) E(\tau)\dd \tau\Big).
\end{align}



Therefore, noticing that $u=\mathbb{P}u-\nabla\Lambda^{-1} v$ and $U_\alpha\le CU_0^{1-\alpha}U_1^{\alpha}$,
we obtain the inequality \eqref{eq.U0} by combining \eqref{eq.Pu.L2.decay1} with \eqref{eq.sigma.L2.decay1}.
\vskip1mm

\noindent \textbf{Step 2}: Estimation of $U_1(t)$. We shall show that
\begin{align}\label{es:U1}
  U_1(t)\le C\Vert (a_0,u_0)\Vert_{ L^1\cap H^1 } + C\int_0^tW_2(t,\tau)U_1(\tau)E(\tau)\dd \tau,
\end{align}
where the weight function $W_2$ is given by
\begin{align*}
  W_2(t,\tau) \triangleq \langle t-\tau\rangle^{-\min\{\frac{N+2\alpha}{2\alpha},\frac{N+2}{2(2-\alpha)}\}}\langle \tau\rangle^{-\frac{N+2}{2(2-\alpha)}}
  \langle t\rangle^{\frac{N+2}{2(2-\alpha)}}.
\end{align*}

We first consider the $\dot H^1$-estimate of $\mathbb{P}u$ which satisfies equation \eqref{eq.Pu.decay}.
It follows from \eqref{eq:heat.decay} that
\begin{align}\label{eq.Pu1.H1.decay}
  \Vert \Lambda e^{-\mu t\Lambda^{\alpha}}\mathbb{P}u_0\Vert_{L^2}
  &\le C\langle t\rangle^{-\frac{N+2}{2\alpha}}\Vert u_0\Vert_{L^1\cap \dot{H}^1}.
\end{align}
By virtue of Lemmas \ref{lem:Leibniz} and \ref{lem:prod-es}, we have
\begin{align}\label{eq:Pu.H1.decay}
  \Vert \widetilde{H}\Vert_{\dot{H}^1}
  &\le C(\Vert\Lambda(u\Lambda^{\alpha}a)\Vert_{L^2}+\Vert\Lambda^{1+\alpha}(a u)\Vert_{L^2}
  +\Vert u\cdot\nabla u\Vert_{\dot{H}^1})\notag\\
  &\le C(\Vert u\Vert_{L^\infty}\Vert a\Vert_{\dot{H}^{1+\alpha}}+\Vert u\Vert_{\dot{H}^{1+\alpha}}
  \Vert a\Vert_{L^\infty}+\Vert u\Vert_{L^\infty}\Vert u\Vert_{\dot{H}^{2}}).
\end{align}
For $N\ge3$, by using the interpolation inequality (e.g. see \cite[Proposition 2.22]{bahouri2011fourier})
\begin{align*}
  \Vert u\Vert_{L^\infty}\le C\Vert u\Vert_{\dot{H}^{s_1}}^{\frac{s_2-N/2}{s_2-s_1}}
  \Vert u\Vert_{\dot{H}^{s_2}}^{\frac{N/2-s_1}{s_2-s_1}},
  \qquad \textrm{for}\;0\le s_1< \tfrac{N}{2}<s_2,
\end{align*} we have
\begin{align*}
  \Vert u\Vert_{L^\infty(\R^N)} \Vert a\Vert_{\dot{H}^{1+\alpha}(\R^N)}
  & \le C\Vert u\Vert_{\dot{H}^1}^{\frac{2}{N}}\Vert u\Vert_{\dot{H}^{N/2+1}}^{1-\frac{2}{N}}\Vert a\Vert_{\dot{H}^{1}}^{1-\frac{2\alpha}{N}}
  \Vert a\Vert_{\dot{H}^{N/2+1}}^{\frac{2\alpha}{N}}   \\
  & \leq C \Big(\|u\|_{\dot H^1}^{\frac{2\alpha}{N}} \|a\|_{\dot H^1}^{1-\frac{2\alpha}{N}} \Big)
  \Big(\|u\|_{\dot H^1}^{\frac{2-2\alpha}{N}} \|u\|_{\dot H^{\frac{N}{2}+1}}^{1-\frac{2}{N}} \|a\|_{\dot H^{\frac{N}{2}+1}}^{\frac{2\alpha}{N}} \Big) \\
  & \le C\Vert (a,u)\Vert_{\dot{H}^1(\R^N)}\Vert (a,u)\Vert_{{H}^s(\R^N)}.
\end{align*}
Similarly, we can estimate the other terms in the right side of \eqref{eq:Pu.H1.decay} and get
\begin{align}\label{eq:Pu.H1.decay.1}
  \Vert \widetilde{H}\Vert_{\dot{H}^1(\R^N)}
  &\le C\Vert (a,u)\Vert_{\dot{H}^1(\R^N)} \Vert (a,u)\Vert_{{H}^s(\R^N)}.
\end{align}
For $N=2$, we claim that
\begin{align*}
  \Vert u\Vert_{L^\infty(\R^2)}\le C\Vert u\Vert_{\dot{H}^1(\R^2)}^{1-\epsilon}\Vert u\Vert_{{H}^s(\R^2)}^{\epsilon},
  \quad \forall 0<\epsilon\ll 1.
\end{align*}
In fact, using the interpolation inequality and taking $\theta= \frac{(2s-2)\epsilon}{s}\in (0,1)$, we have
\begin{align*}
  \Vert u\Vert_{L^\infty(\R^2)}&\le C \Vert u\Vert_{\dot{H}^{1-\theta}(\R^2)}^{1/2}\Vert u\Vert_{\dot{H}^{1+\theta}(\R^2)}^{1/2}\\
  &\le C \Big (\Vert u\Vert_{L^2}^\theta \Vert u\Vert_{\dot{H}^1}^{1-\theta}\Big)^{1/2}
  \Big(\Vert u\Vert_{\dot{H}^1}^{1-\frac{\theta}{s-1}} \Vert u\Vert_{\dot{H}^s}^{\frac{\theta}{s-1}}\Big)^{1/2}\\
  &\le C \Vert u\Vert_{{H}^s(\R^2)}^{\frac{\theta}{2}+\frac{\theta}{2s-2}} \Vert u\Vert_{\dot{H}^1(\R^2)}^{1-\frac{\theta}{2}-\frac{\theta}{2s-2}}
  =C\Vert u\Vert_{\dot{H}^1(\R^2)}^{1-\epsilon}\Vert u\Vert_{{H}^s(\R^2)}^\epsilon.
\end{align*}
Thus, it is immediate to show that for every $0<\epsilon <1-\alpha$ and $N=2$,
\begin{align*}
  \|u\|_{L^\infty(\R^2)} \|a\|_{\dot H^{1+\alpha}(\R^2)}
  & \leq C \Big(\|u\|_{\dot H^1}^\alpha \|a\|_{\dot H^1}^{1-\alpha}\Big) \Big(\|u\|_{\dot H^1}^{1-\epsilon -\alpha} \|u\|_{H^s}^\epsilon
  \|a\|_{\dot H^{N/2+1}}^\alpha \Big) \\
  & \leq C \|(u,a)\|_{\dot H^1(\R^2)} \|(u,a)\|_{H^s(\R^2)},
\end{align*}
and by treating the remaining terms in a similar way,
we can also get the estimate \eqref{eq:Pu.H1.decay.1} for $N=2$.
On the other hand, we split $\widetilde{H}$ into $\widetilde{H}=\widetilde{H}_1+\widetilde{H}_2$, with
\begin{align}\label{def:tildH1H2}
  \widetilde{H}_1\triangleq \mu \big( u\Lambda^{\alpha}a-\Lambda^{\alpha}(a u)\big),\quad \textrm{and}\quad
  \widetilde{H}_2 \triangleq -u\cdot\nabla u.
\end{align}
By virtue of Lemmas \ref{lem:Leibniz} and the interpolation inequality, we infer that
\begin{align}\label{eq:LambdaPu.L1.decay}
  \Vert \Lambda^{1-\alpha}\widetilde{H}_1\Vert_{L^1}
  &\le C\Vert\Lambda^{1-\alpha}(u\Lambda^{\alpha}a)\Vert_{L^1}+ C\Vert\Lambda(a u)\Vert_{L^1}\notag\\
  &\le C \Big(\Vert u\Vert_{\dot{H}^{1-\alpha}}\Vert a\Vert_{\dot{H}^{\alpha}}+\Vert u\Vert_{L^2}\Vert a\Vert_{\dot{H}^{1}}
  +\Vert u\Vert_{\dot{H}^{1}}\Vert a\Vert_{L^2}\Big) \nonumber \\
  & \le C\Vert(a,u)\Vert_{L^2}\Vert (a,u)\Vert_{\dot{H}^{1}},
\end{align}
and
\begin{align}\label{eq:tildeH2.L1.decay}
  \Vert \widetilde{H}_2\Vert_{L^1} = \|u\cdot \nabla u\|_{L^1}
  \le \Vert u\Vert_{L^2}\Vert u\Vert_{\dot{H}^{1}}.
\end{align}
Using the above inequalities \eqref{eq:Pu.H1.decay.1}, \eqref{eq:LambdaPu.L1.decay},
\eqref{eq:tildeH2.L1.decay} and \eqref{eq.tildeH.L1}, we argue as \eqref{eq.Pu2.decay} to derive
\begin{align*}
&\quad\int_0^t\Vert \Lambda e^{-\mu (t-\tau)\Lambda^{\alpha}}  \mathbb{P}\widetilde{H}(\tau)\Vert_{L^2}\dd \tau\\
  &\le C\int_0^t\langle t-\tau\rangle^{-\frac{N+2\alpha}{2\alpha}}
  \Vert \Lambda^{1-\alpha}\widetilde{H}_1(\tau)\Vert_{\dot{H}^{\alpha}\cap L^1}\dd \tau+C\int_0^t\langle t-\tau\rangle^{-\frac{N+2}{2\alpha}}
  \Vert \widetilde{H}_2(\tau)\Vert_{\dot{H}^{1}\cap L^1}\dd \tau\notag\\
  &\le C\int_0^t\langle t-\tau\rangle^{-\min\big\{\frac{N+2\alpha}{2\alpha},\frac{N+2}{2(2-\alpha)}\big\}}
  \langle \tau\rangle^{-\frac{N+2}{2(2-\alpha)}}U_1(\tau)E(\tau)\dd \tau.
\end{align*}
Combining the above estimate with \eqref{eq.Pu1.H1.decay}, \eqref{eq.Pu.decay} leads to
\begin{align}\label{eq:Pu-H1decay}
  \|\mathbb{P}u(t)\|_{\dot H^1} \leq  C\langle t\rangle^{-\frac{N+2}{2\alpha}}\Vert u_0\Vert_{L^1\cap \dot{H}^1}
  + C\langle t\rangle^{-\frac{N+2}{2(2-\alpha)}}
  \int_0^tW_2(t,\tau) U_1(\tau)E(\tau)\dd \tau.
\end{align}

Next we consider the $\dot H^1$-estimate of compressible part $(a,v)$ which satisfies \eqref{eq.CP}.
Thanks to Lemma \ref{lem:GM.L2}, we have
\begin{align*}
  \Vert\vert\xi\vert\widehat{\mathcal{G}}_{11}(t)\widehat{a_0}\Vert_{L^2}
  \le C \langle t\rangle^{-\frac{N+2}{2(2-\alpha)}}\Vert a_0\Vert_{L^1\cap \dot{H}^1},
\end{align*}
\begin{align*}
  \Vert\vert\xi\vert\widehat{\mathcal{G}}_{12}(t)(\widehat{v_0},\widehat{a_0})\Vert_{L^2}
  \le \Vert\vert\xi\vert\widehat{\mathcal{G}}_{12}(t) (\widehat{u_0},\widehat{a_0}) \Vert_{L^2}
  \le C \langle t\rangle^{-\frac{N+2(2-\alpha)}{2(2-\alpha)}}\Vert (u_0, a_0)\Vert_{L^1\cap \dot{H}^1},
\end{align*}
\begin{align*}
  \Vert\vert\xi\vert\widehat{\mathcal{G}}_{22}(t)\widehat{v_0}\Vert_{L^2}
  \le \Vert\vert\xi\vert\widehat{\mathcal{G}}_{22}(t)\widehat{u_0}\Vert_{L^2}
  \le C \langle t\rangle^{-\frac{N+2(2-\alpha)}{2(2-\alpha)}}\Vert u_0\Vert_{L^1\cap \dot{H}^1}.
\end{align*}
Recall that $F= \Div (a\, u)$ and $G = \Lambda^{-1} \Div \widetilde{G}$ with $\widetilde{G}$ given by
\eqref{def:tild-G}. By applying Lemmas \ref{lem:Leibniz}, \ref{lem:prod-es} and \ref{Lem:composite},
and arguing as estimating the terms on the right-hand side of \eqref{eq:Pu.H1.decay}, we see that
\begin{align}\label{eq.F.H1.Decay}
  \Vert F\Vert_{\dot{H}^1}
  &\le C \big(\Vert u\cdot\nabla a\Vert_{\dot{H}^1}+\Vert a\Div u\Vert_{\dot{H}^1}\big)\notag\\
  &\le C \big(\Vert u\Vert_{L^\infty}\Vert a\Vert_{\dot{H}^2}+\Vert a\Vert_{L^\infty}\Vert u\Vert_{\dot{H}^2}\big)
  \le C\Vert (a,u)\Vert_{\dot{H}^1}\Vert (a,u)\Vert_{{H}^s},
\end{align}
and
\begin{align}\label{eq.tildeG.H1.Decay}
  \Vert \widetilde{G}\Vert_{\dot{H}^1}
  &\le C \Vert \widetilde{H}\Vert_{\dot{H}^1}+C\Vert \big((a+1)^{\gamma-2}-1\big)\nabla a\Vert_{\dot{H}^1}\notag\\
  &\le C \Big(\Vert u\Vert_{L^\infty}\Vert a\Vert_{\dot{H}^{1+\alpha}}
  +\Vert u\Vert_{\dot H^{1+\alpha}}\Vert a\Vert_{L^\infty}
  +\Vert u\Vert_{L^\infty}\Vert u\Vert_{\dot{H}^{2}}+\Vert a\Vert_{L^\infty}\Vert a\Vert_{\dot{H}^{2}} + \|a\|_{\dot H^1}^2 \Big)\notag\\
  &\le C\Vert (a,u)\Vert_{\dot{H}^1}\Vert (a,u)\Vert_{{H}^s}.
\end{align}
We also split $\widetilde{G}$ into $\widetilde{G}=\widetilde{H}_1+\widetilde{G}_2$, with $\widetilde{H}_1$ given by \eqref{def:tildH1H2} and
$\widetilde{G}_2\triangleq - u\cdot \nabla u -\gamma\big((1+a)^{\gamma-2}-1\big)\nabla a$, and we easily find that
\begin{align}\label{eq:tildeG2.L1.decay}
  \Vert \widetilde{G}_2\Vert_{L^1}
  \le C\big(\Vert u\Vert_{L^2}\Vert u\Vert_{\dot{H}^{1}}+\Vert a\Vert_{L^2}\Vert a\Vert_{\dot{H}^{1}}\big).
\end{align}
The inequalities \eqref{eq.F.H1.Decay}-\eqref{eq:tildeG2.L1.decay} and  \eqref{eq:LambdaPu.L1.decay} combined with Lemma \ref{lem:GM.L2} gives
\begin{align*}
  & \quad \int_0^t\Vert \big(\vert\xi\vert\widehat{\mathcal{G}}_{11}(t-\tau)\widehat{F}(\tau),
  \vert \xi\vert \widehat{\mathcal{G}}_{21}(t-\tau)\widehat{F}(\tau)\big)\Vert_{L^2}\dd \tau
  \le C\int_0^t\langle t-\tau\rangle^{-\frac{N+2}{2(2-\alpha)}}\Vert F(\tau)\Vert_{\dot{H}^1\cap L^1}\dd \tau\notag\\
  & \leq C \int_0^t\langle t-\tau\rangle^{-\frac{N+2}{2(2-\alpha)}} \Big(\|(a,u)\|_{L^2}
  \|(\nabla a,\nabla u)\|_{L^2} + \Vert (a,u)\Vert_{\dot{H}^1}\Vert (a,u)\Vert_{{H}^s}\Big) \dd \tau \\
  &\le C\int_0^t\langle t-\tau\rangle^{-\frac{N+2}{2(2-\alpha)}}\langle \tau\rangle^{-\frac{N+2}{2(2-\alpha)}}U_1(\tau)E(\tau)\dd \tau,
\end{align*}
and
\begin{align*}
  & \quad \int_0^t\Vert \big(\vert\xi\vert\widehat{\mathcal{G}}_{12}(t-\tau)\widehat{G}(\tau),
  \vert\xi\vert\widehat{\mathcal{G}}_{22}(t-\tau)\widehat{G}(\tau) \big)\Vert_{L^2}\dd \tau\\
  &\le C\int_0^t\langle t-\tau\rangle^{-\frac{N+2}{2(2-\alpha)}}
  \Vert \Lambda^{1-\alpha}\widetilde{H}_1(\tau)\Vert_{\dot{H}^{\alpha}\cap L^1}\dd \tau
  +C\int_0^t\langle t-\tau\rangle^{-\frac{N+4-2\alpha}{2(2-\alpha)}}
  \Vert \widetilde{G}_2(\tau)\Vert_{\dot{H}^1\cap L^1}\dd \tau\notag\\
  & \leq C\int_0^t\langle t-\tau\rangle^{-\frac{N+2}{2(2-\alpha)}}
  \Big(\|(a,u)\|_{L^2} \|( a, u)\|_{\dot{H}^1}
  + \Vert (a,u)\Vert_{\dot{H}^1}\Vert (a,u)\Vert_{{H}^s} \Big) \dd \tau \\
  &\le C\int_0^t\langle t-\tau\rangle^{-\frac{N+2}{2(2-\alpha)}}
  \langle \tau\rangle^{-\frac{N+2}{2(2-\alpha)}}U_1(\tau)E(\tau)\dd \tau.
\end{align*}
Gathering the above estimates yields
\begin{align}\label{eq.sigma.H1.decay1}
  \Vert (a,v)(t)\Vert_{\dot H^1}
  \le C\langle t\rangle^{-\frac{N+2}{2(2-\alpha)}}
  \Big(\Vert (a_0,u_0)\Vert_{L^1\cap \dot H^1} + \int_0^tW_2(t,\tau)U_1(\tau)E(\tau)\dd \tau\Big).
\end{align}

Hence, combining \eqref{eq:Pu-H1decay} with \eqref{eq.sigma.H1.decay1} completes the proof of \eqref{es:U1}.
\vskip1mm

\noindent \textbf{Step 3}: Estimation of $V(t)$. We will prove that
\begin{align}\label{es:Vt}
  V(t)\le C\Vert (a_0,u_0)\Vert_{H^s \cap L^1} + C \big(U_0+U_1\big)^2(t) + C\int_0^tW_3(t,\tau) E(\tau) V(\tau)\dd \tau,
\end{align}
where the weight function $W_3$ is given by
\begin{align*}
  W_3(t,\tau) \triangleq \langle t-\tau\rangle^{-\min\big\{\frac{N}{\alpha},\frac{N}{2-\alpha}\big\}}
  \langle \tau\rangle^{-\frac{N}{2-\alpha}} \langle t\rangle^{\frac{N}{2-\alpha}}.
\end{align*}

For $\mathbb{P}u$ solving the equation \eqref{eq.Pu.decay}, in view of Lemma \ref{lem:heat.decay.L1}, we have
\begin{align*}
  \Vert e^{-\mu t\vert\xi\vert^{\alpha}}\widehat{\mathbb{P}u_0}\Vert_{L^1}
  \le C\langle t\rangle^{-\frac{N}{\alpha}} \big(\Vert {u_0}\Vert_{L^1}+\Vert\widehat{u_0}\Vert_{L^1} \big).
\end{align*}
Taking advantage of Young's inequality, we estimate $\widetilde{H}$ (given by \eqref{def:tild-H}) as follows:
\begin{align}\label{eq:FHtild-L1}
  \Vert \mathcal{F}(\widetilde{H}) \Vert_{L^1}
  &\le C \big(\Vert\widehat{u\Lambda^{\alpha}a}(\xi)-\vert\xi\vert^\alpha\widehat{a u}(\xi)\Vert_{L^1}
  +\Vert \widehat{u\cdot\nabla u}\Vert_{L^1} \big) \notag \\
  &\le C \bigg( \Big\Vert\int_{\R^N} \widehat{u}(\xi-\eta)(\vert\eta\vert^\alpha
  -\vert\xi\vert^\alpha)\widehat{a}(\eta)\dd \eta \Big\Vert_{L^1_\xi}
  + \Vert \widehat{u}*\widehat{\nabla u}\Vert_{L^1} \bigg) \notag \\
  &\le C \bigg( \Big\Vert\int_{\R^N} \widehat{u}(\xi-\eta)
  \vert\xi -\eta\vert^\alpha \widehat{a}(\eta)\dd \eta \Big\Vert_{L^1_\xi}
  +\Vert \widehat{u}\Vert_{L^1}\Vert \vert\xi\vert \widehat{u}(\xi)\Vert_{L^1_\xi} \bigg) \notag \\
  & \le C \big(\Vert\vert\xi\vert^\alpha\widehat{u}\Vert_{L^1}\Vert \widehat{a} \Vert_{L^1}
  + \Vert \widehat{u}\Vert_{L^1}\Vert \vert\xi\vert \widehat{u}\Vert_{L^1}\big)
  \le C (\Vert \widehat{a}\Vert_{L^1}+\Vert \widehat{u}\Vert_{L^1})\Vert {u}\Vert_{H^s},
\end{align}
where in the last inequality we have used the fact that
\begin{align}\label{eq:fact}
  \|(1+|\xi|)\widehat{u}\|_{L^1} \leq C \|(1+|\xi|)^{1-s}\|_{L^2} \|u\|_{H^s} \leq C \|u\|_{H^s}.
\end{align}
By using\eqref{eq.tildeH.L1}, \eqref{eq:FHtild-L1} and Lemma \ref{lem:heat.decay.L1}, we find
\begin{align*}
  \int_0^t\Vert e^{-\mu (t-\tau)\vert\xi\vert^{\alpha}}\widehat{\mathbb{P}\widetilde{H}}(\tau)\Vert_{L^1}\dd \tau
  &\le C\int_0^t\langle t-\tau\rangle^{-\frac{N}{\alpha}} \big(\Vert \widetilde{H}(\tau)\Vert_{L^1}
  +\Vert \mathcal{F}(\widetilde{H})(\tau)\Vert_{L^1}\big) \dd \tau\notag\\
  &\le C\int_0^t\langle t-\tau\rangle^{-\frac{N}{\alpha}}
  \langle \tau\rangle^{-\frac{N}{2-\alpha}}\big(U_0U_1+U_0 U_\alpha+E V\big)(\tau)\dd \tau .
\end{align*}
Combining the above estimates with \eqref{eq.Pu.decay} leads to
\begin{align*}
  \|\widehat{\mathbb{P} u}(t)\|_{L^1} \leq
  C \langle t\rangle^{-\frac{N}{2-\alpha}}
  \Big( \|u_0\|_{L^1\cap H^s} +  \int_0^tW_3(t,\tau)\big(U_0U_1+U_0U_\alpha+E V\big)(\tau)\dd \tau\Big).
\end{align*}

Let us consider the estimation of $(a,v)$ which solves \eqref{eq.CP}.
For the term of initial data, by virtue of Lemma \ref{lem:GM.L1}, we get
\begin{align*}
  \Vert\widehat{\mathcal{G}}_{11}(t)\widehat{a_0}\Vert_{L^1}
  \le C \langle t\rangle^{-\frac{N}{2-\alpha}} \big(\Vert a_0\Vert_{L^1}+\Vert\widehat{a_0}\Vert_{L^1}\big),
\end{align*}
\begin{align*}
  \Vert\widehat{\mathcal{G}}_{12}(t)\widehat{v_0}\Vert_{L^1} +\Vert\widehat{\mathcal{G}}_{21}(t)\widehat{a_0}\Vert_{L^1}
  \le C \langle t\rangle^{-\frac{N+1-\alpha}{2-\alpha}} \big(\Vert (a_0,u_0)\Vert_{L^1}
  + \Vert(\widehat{a_0},\widehat{u_0})\Vert_{L^1}\big),
\end{align*}
\begin{align*}
  \Vert\widehat{\mathcal{G}}_{22}(t)\widehat{v_0}\Vert_{L^1}
  \le C\Vert\widehat{\mathcal{G}}_{22}(t)\widehat{u_0}\Vert_{L^1}
  \le C \langle t\rangle^{-\frac{N+1-\alpha}{2-\alpha}} \big(\Vert u_0\Vert_{L^1}+\Vert\widehat{u_0}\Vert_{L^1}\big).
\end{align*}
Owing to Young's inequality and \eqref{eq:fact}, we see that $F= \Div(a\,u)$ satisfies
\begin{align}\label{eq.hatF.L1.Decay}
  \Vert \widehat{F}\Vert_{L^1}
  &\le C \big(\Vert \widehat{u\cdot\nabla a}\Vert_{L^1}+\Vert \widehat{a\Div u}\Vert_{L^1}\big)
  \leq C \big(\Vert \widehat{u}\Vert_{L^1}\Vert \vert\xi\vert\widehat{a}\Vert_{L^1}
  +\Vert  \vert\xi\vert  \widehat{u}\Vert_{L^1}\Vert\widehat{a}\Vert_{L^1}\big)\notag\\
  &\leq C \Vert (\widehat{u},\widehat{a})\Vert_{L^1} \Vert (a, u)\Vert_{H^s}.
\end{align}
Similarly, noting that $\Vert \mathcal{F}((a+1)^{\gamma-2}-1)\Vert_{L^1}\le C\Vert \widehat{a}\Vert_{L^1}$
(from Lemma \ref{Lem:composite}), the term $G$ given by \eqref{def:G} has
\begin{align}\label{eq.hatG.L1.Decay}
  \Vert \widehat{G}\Vert_{L^1}
  &\le C \big(\Vert\vert\xi\vert^\alpha\widehat{u}\Vert_{L^1}\Vert \widehat{a}\Vert_{L^1}+\Vert \widehat{u}\Vert_{L^1}
  \Vert \vert\xi\vert \widehat{u}\Vert_{L^1}+\Vert \widehat{a}\Vert_{L^1}\Vert \vert\xi\vert \widehat{a}\Vert_{L^1}\big)\notag \\
  &\le C\Vert (\widehat{u},\widehat{a})\Vert_{L^1} \Vert (a, u)\Vert_{H^s}.
\end{align}
The above inequalities combined with  \eqref{eq.tildeG.L1.Decay} and Lemma \ref{lem:GM.L1} give
\begin{align*}
  &\quad\int_0^t\Vert \big(\widehat{\mathcal{G}}_{11}(t-\tau)\widehat{F}(\tau),
  \widehat{\mathcal{G}}_{21}(t-\tau)\widehat{F}(\tau)\big)\Vert_{L^1}\dd \tau\\
  &\le C\int_0^t\langle t-\tau\rangle^{-\frac{N}{2-\alpha}}(\Vert F(\tau)\Vert_{L^1}
  +\Vert\widehat{F}(\tau)\Vert_{L^1})\dd \tau\notag\\
  &\le C\int_0^t\langle t-\tau\rangle^{-\frac{N}{2-\alpha}}\langle \tau\rangle^{-\frac{N}{2-\alpha}}\big(U_0U_1+E V\big)(\tau)\dd \tau,
\end{align*}
and
\begin{align*}
 &\quad \int_0^t\Vert \big(\widehat{\mathcal{G}}_{12}(t-\tau)\widehat{G}(\tau),
  \widehat{\mathcal{G}}_{22}(t-\tau)\widehat{G}(\tau)\big)\Vert_{L^1}\dd \tau\\
  &\le C\int_0^t\langle t-\tau\rangle^{-\frac{N+1-\alpha}{2-\alpha}}
  (\Vert \widetilde{G}(\tau)\Vert_{L^1}+\Vert\widehat{G}(\tau)\Vert_{L^1})\dd \tau\notag\\
  &\le C\int_0^t\langle t-\tau\rangle^{-\frac{N}{2-\alpha}}\langle \tau\rangle^{-\frac{N}{2-\alpha}}\big(U_0U_1+U_0U_\alpha+E V\big)(\tau)\dd \tau.
\end{align*}

Hence, gathering the above estimates completes the proof of \eqref{es:Vt}.
\vskip1mm

\noindent \textbf{Step 4}: Proof of the decay estimates \eqref{eq.decay1}-\eqref{eq.decay3}.
According to \eqref{eq.U0} and \eqref{es:U1}, we have
\begin{align}
  (U_0+U_1)(t) \le C\Vert (a_0,u_0)\Vert_{H^s\cap L^1} + C\int_0^t(W_1(t,\tau)+W_2(t,\tau))E(\tau)(U_0+U_1)(\tau)\dd\tau.
\end{align}
Gronwall's inequality guarantees that
\begin{align*}
  (U_0+U_1)(t)\leq C\exp\Big(E(t)\int_0^t(W_1(t,\tau)+W_2(t,\tau))\dd\tau\Big).
\end{align*}
By virtue of Lemma \ref{lem:convolution} and noting that $\frac{N+2\alpha}{2\alpha}\geq \frac{n+2}{2(2-\alpha)}$ for every $\alpha\in(0,1]$
and $N\geq 2$, we have
\begin{align*}
  \sup_{0\le t<+\infty}\int_0^t \big(W_1(t,\tau)+W_2(t,\tau)\big)\dd\tau<+\infty.
\end{align*}
Therefore, we see that $(U_0+U_1)(t) <\infty$ for all $t\in \R_+$.
Inserting this estimate into \eqref{es:Vt} and using Gronwall's inequality, Lemma \ref{lem:convolution} again and the fact that $\Vert f\Vert_{L^\infty}\le \Vert \widehat{f}\Vert_{L^1}$,
we conclude the decay results \eqref{eq.decay1}-\eqref{eq.decay3}.
\end{proof}

\subsection{Refined upper bounds for the velocity field}
In this subsection we shall prove a refined decay estimate for the velocity field $u$.

\begin{proposition}\label{Prop:upper.refined}
Under the assumptions of Proposition \ref{Prop:upper.rough}, we have
\begin{equation}\label{eq.decay.refined}
\begin{aligned}
  &\Vert u(t)\Vert_{L^2}\le C\langle t\rangle^{-\frac{N+2(1-\alpha)}{2(2-\alpha)}},\quad
  \Vert \Lambda^\alpha u(t)\Vert_{L^2}\le C\langle t\rangle^{-\frac{N+2}{2(2-\alpha)}}.
\end{aligned}
\end{equation}
In addition, we have for $\alpha\in[\frac{2N}{3N+2},1]$,
\begin{align}\label{eq.decay.refined2}
  &\Vert \mathbb{P}u(t) \Vert_{L^2} \le C \langle t\rangle^{-\frac{N}{2\alpha}}.
\end{align}
\end{proposition}

\begin{proof}[Proof of Proposition \ref{Prop:upper.refined}]
We recalculate the estimates \eqref{eq.Pu2.decay}, \eqref{eq.G21.F.decay} and \eqref{eq.G12.G.decay} as follows
\begin{align}\label{eq.Pu2.decay.R1}
  & \int_0^t\Vert e^{-\mu (t-\tau)\Lambda^{\alpha}} \mathbb{P} \widetilde{H}(\tau)\Vert_{L^2}\dd \tau
  \le C\int_0^t\langle t-\tau\rangle^{-\frac{N}{2\alpha}} \Vert \widetilde{H}(\tau)\Vert_{L^1 \cap L^2} \dd \tau\notag\\
  & \leq C \int_0^t\langle t-\tau\rangle^{-\frac{N}{2\alpha}} \Big( \|\Lambda^\alpha u\|_{L^2} \big(\|a\|_{L^2}
  + \|\widehat{a}\|_{L^1}\big) + \|\nabla u\|_{L^2} \big( \|u\|_{L^2} + \|\widehat{u}\|_{L^1}\big) \Big)\dd\tau \nonumber \\
  &\le C\int_0^t\langle t-\tau\rangle^{-\frac{N}{2\alpha}}
  \langle \tau \rangle^{-\frac{N}{2-\alpha}}  (U_0U_\alpha+U_0U_1 + E V)(\tau)   \dd \tau \notag \\
  &\le C\langle t\rangle^{-\min \big\{\frac{N}{2\alpha},\frac{N}{2-\alpha}\big\}} \big( U_0U_\alpha + U_0U_1 + E V \big)(t),
\end{align}
and
\begin{align}\label{eq.G21.F.decay.R1}
  \int_0^t\Vert\widehat{\mathcal{G}}_{21}(t-\tau)\widehat{F}(\tau) \Vert_{L^2}\dd \tau
  &\le C\int_0^t\langle t-\tau\rangle^{-\frac{N+2(2-\alpha)}{2(2-\alpha)}}
  \Vert (au)(\tau)\Vert_{L^1\cap\dot{H}^1}\dd \tau\notag\\
  &\le C\int_0^t\langle t-\tau\rangle^{-\frac{N+2(2-\alpha)}{2(2-\alpha)}}
  \langle \tau\rangle^{-\frac{N}{2-\alpha}}\dd \tau \big(U_0^2 + E V\big)(t)\notag \\
  &\le C\langle t\rangle^{-\frac{N+2(2-\alpha)}{2(2-\alpha)}} \big(U_0^2+E V\big)(t),
\end{align}
and
\begin{align}\label{eq.G12.G.decay.R1}
  &\int_0^t\Vert \widehat{\mathcal{G}}_{12}(t-\tau)\widehat{G}(\tau)\Vert_{L^2}\dd \tau
  \le C\int_0^t\langle t-\tau\rangle^{-\frac{N+2(1-\alpha)}{2(2-\alpha)}}
  \Vert \widetilde{G}(\tau)\Vert_{L^1\cap L^2} \dd \tau\notag\\
  & \leq C\int_0^t\langle t-\tau\rangle^{-\frac{N+2(1-\alpha)}{2(2-\alpha)}}
  \Big(\|\Lambda^\alpha u\|_{L^2} \|a\|_{L^2} + \|(a,u)\|_{L^2} \|(\nabla a,\nabla u)\|_{L^2}
  + E(\tau) \|(\widehat{a},\widehat{u})\|_{L^1} \Big) \dd \tau\nonumber \\
  &\le C\int_0^t\langle t-\tau\rangle^{-\frac{N+2(1-\alpha)}{2(2-\alpha)}}
  \langle \tau\rangle^{-\frac{N}{2-\alpha}}\dd \tau\, \big(U_0U_1+ U_0U_\alpha+E V\big)(t)\notag \\
  & \le C\langle t\rangle^{-\frac{N+2(1-\alpha)}{2(2-\alpha)}}
  \big(U_0U_1+U_0U_{\alpha}+E V\big)(t).
\end{align}
Noting that $\frac{N+2(1-\alpha)}{2(2-\alpha)} \leq \min\{\frac{N}{2\alpha}, \frac{N}{2-\alpha}\}$,
we can gather the estimates \eqref{eq.Pu1.decay}, \eqref{eq.G12.int.decay}, \eqref{eq.G22.int.decay},  \eqref{eq.Pu2.decay.R1}, \eqref{eq.G21.F.decay.R1} and \eqref{eq.G12.G.decay.R1} to obtain
\begin{align}\label{eq.tildeU0}
  \sup_{0\le \tau\le t}  \langle \tau\rangle^{\frac{N+2(1-\alpha)}{2(2-\alpha)}}\Vert u(\tau)\Vert_{L^2}
  \le C\Vert (a_0,u_0)\Vert_{ L^1\cap L^2} + C(U_0+U_1)^2(t) + CE(t) V(t).
\end{align}

Now denoting by
\begin{align*}
  &\widetilde{U}_\alpha(t)\triangleq \sup_{0\le \tau\le t}
  \langle \tau\rangle^{\frac{N+2}{2(2-\alpha)}}\Vert \Lambda^\alpha u(\tau)\Vert_{L^2},
\end{align*}
we intend to show
\begin{align}\label{eq.tildeUalpha}
  \widetilde{U}_\alpha(t)
  \le C\Vert (a_0,u_0)\Vert_{H^s\cap L^1} + C(U_0+U_1)^2(t) + C E(t) V(t).
\end{align}
Indeed,
using inequalities \eqref{eq.tildeH.L1} and \eqref{eq:Pu.H1.decay} yields
\begin{align}\label{eq.Pu2.Halpha.decay}
  \int_0^t\Vert \Lambda^\alpha e^{-\mu (t-\tau)\Lambda^{\alpha}} \mathbb{P} \widetilde{H}(\tau)\Vert_{L^2}\dd \tau
  &\le C\int_0^t\langle t-\tau\rangle^{-\frac{N+2\alpha}{2\alpha}}\Vert \widetilde{H}(\tau)\Vert_{\dot{H}^1\cap L^1}\dd \tau\notag\\
  &\le C\int_0^t\langle t-\tau\rangle^{-\frac{N+2}{2(2-\alpha)}}
  \langle \tau\rangle^{-\frac{N}{2-\alpha}}\dd \tau \,\big(U_0U_1+U_0U_\alpha+E V\big)(t) \notag \\
  &\le C\langle t\rangle^{-\frac{N+2}{2(2-\alpha)}} \big(U_0 U_1+U_0U_\alpha+E V\big)(t);
\end{align}
and note that Lemma \ref{eq:heat.decay} implies
\begin{align*}
  \|\Lambda^\alpha e^{-\mu t\Lambda^\alpha} \mathbb{P} u_0\|_{L^2} \leq C \langle t\rangle^{-\frac{N+ 2\alpha}{2\alpha}}
  \|u_0\|_{\dot H^\alpha \cap L^1},
\end{align*}
which in combination with \eqref{eq.Pu.decay} and \eqref{eq.Pu2.Halpha.decay} leads to
\begin{align}\label{es:Lamb-Pu-L2}
  \|\Lambda^\alpha \mathbb{P}u(t)\|_{L^2} \leq
  C\langle t\rangle^{-\frac{N+2}{2(2-\alpha)}} \Big( \|u_0\|_{H^s\cap L^1} + \big(U_0 U_1+U_0U_\alpha+E V\big)(t)\Big).
\end{align}
For the estimation of $v$ satisfying \eqref{eq.CP}, Lemma \ref{lem:GM.L2} ensures that
\begin{align*}
  \Vert\vert\xi\vert^\alpha \widehat{\mathcal{G}}_{12}(t)\widehat{a_0}\Vert_{L^2}
  + \Vert\vert\xi\vert^\alpha \widehat{\mathcal{G}}_{22}(t)\widehat{v_0}\Vert_{L^2}
  \le C \langle t\rangle^{-\frac{N+2}{2(2-\alpha)}}\Vert (u_0, a_0)\Vert_{L^1\cap \dot{H}^\alpha};
\end{align*}
thanks to inequalities \eqref{eq.tildeG.L1.Decay}, \eqref{eq.F.H1.Decay}, \eqref{eq.tildeG.H1.Decay} and Lemmas \ref{lem:convolution}, \ref{lem:GM.L2}, we have that for every $\alpha\in (0,1]$,
\begin{align*}
  \int_0^t\Vert\vert\xi\vert^\alpha\widehat{\mathcal{G}}_{21}(t-\tau)\widehat{F}(\tau)\Vert_{L^2}\dd \tau
  &\le C\int_0^t\langle t-\tau\rangle^{-\frac{N+2}{2(2-\alpha)}}\Vert F(\tau)\Vert_{\dot{H}^1\cap L^1}\dd \tau\notag\\
  &\le C\int_0^t\langle t-\tau\rangle^{-\frac{N+2}{2(2-\alpha)}}\langle \tau\rangle^{-\frac{N}{2-\alpha}}\dd \tau
  \big(U_0U_1+E V\big)(t) \notag\\
  &\le C\langle t\rangle^{-\frac{N+2}{2(2-\alpha)}} \big(U_0U_1+E V\big)(t),
\end{align*}
and
\begin{align*}
  \int_0^t\Vert\vert\xi\vert^\alpha\widehat{\mathcal{G}}_{22}(t-\tau)\widehat{G}(\tau)\Vert_{L^2}\dd \tau
  &\le C\int_0^t\langle t-\tau\rangle^{-\frac{N+2}{2(2-\alpha)}}\Vert \widetilde{G}(\tau)\Vert_{\dot{H}^1\cap L^1}\dd \tau\notag\\
  &\le C\int_0^t\langle t-\tau\rangle^{-\frac{N+2}{2(2-\alpha)}}
  \langle \tau\rangle^{-\frac{N}{2-\alpha}}\dd \tau \big(U_0U_1+U_0U_\alpha +E V\big)(t)\notag \\
  &\le C\langle t\rangle^{-\frac{N+2}{2(2-\alpha)}} \big(U_0U_1+U_0U_\alpha +E V\big)(t);
\end{align*}
thus we find
\begin{align}\label{eq.sigma.Halp.decay}
  \Vert \Lambda^\alpha v(t)\Vert_{L^2}
  \le C\langle t\rangle^{-\frac{N+2}{2(2-\alpha)}}
  \Big(\Vert (a_0,u_0)\Vert_{L^1\cap \dot H^\alpha} + \big(U_0U_\alpha + U_0 U_1+ E V\big)(t)\Big).
\end{align}
Hence we obtain the desired result \eqref{eq.tildeUalpha} by combining the estimates \eqref{es:Lamb-Pu-L2} and \eqref{eq.sigma.Halp.decay}.

Next for every $\alpha\in[\frac{2N}{3N+2},1]$, we show that
\begin{align}\label{eq.Pu.refinied}
  \sup_{0\le \tau\le t}\langle \tau\rangle^{\frac{N}{2\alpha}}\Vert \mathbb{P}u(\tau)\Vert_{L^2}
  \le C\Vert (a_0,u_0)\Vert_{L^2\cap L^1} + C\big(U_0\widetilde{U}_\alpha+U_0U_{1}+VU_1\big)(t).
\end{align}
Indeed, noting that $\frac{N}{2\alpha}\le \frac{N+1}{2-\alpha}$ for every $\alpha\in[\frac{2N}{3N+2},1]$,
and using Lemmas \ref{lem:heat.decay}, \ref{lem:convolution} and \eqref{eq.tildeH.L1}, we get
\begin{align*}
  \int_0^t\Vert e^{-\mu (t-\tau)\Lambda^{\alpha}}H(\tau)\Vert_{L^2}\dd \tau
  &\le C\int_0^t\langle t-\tau\rangle^{-\frac{N}{2\alpha}}\Vert \widetilde{H}(\tau)\Vert_{L^2\cap L^1}\dd \tau\notag\\
  & \le C\int_0^t\langle t-\tau\rangle^{-\frac{N}{2\alpha}}
  \big(\Vert \Lambda^\alpha u\Vert_{L^2} \Vert a\Vert_{L^2\cap L^\infty}
  + \Vert u\Vert_{L^2\cap L^\infty} \Vert\nabla u\Vert_{L^2}\big)\dd \tau \\
  &\le C\int_0^t\langle t-\tau\rangle^{-\frac{N}{2\alpha}}
  \langle \tau\rangle^{-\frac{N+1}{2-\alpha}}\dd \tau \big(U_0\widetilde{U}_\alpha + V \widetilde{U}_\alpha
  + U_0U_1 + VU_1\big)(t) \notag \\
  &\le C\langle t\rangle^{-\frac{N}{2\alpha}} \big(U_0\widetilde{U}_\alpha + V \widetilde{U}_\alpha+U_0U_1 + VU_1\big)(t),
\end{align*}
with combined with the estimate \eqref{eq.Pu1.decay} yields \eqref{eq.Pu.refinied}.

Finally, collecting \eqref{eq.tildeU0}, \eqref{eq.tildeUalpha} and \eqref{eq.Pu.refinied} and using Proposition \ref{Prop:upper.rough},
we deduce the desired estimates \eqref{eq.decay.refined}-\eqref{eq.decay.refined2}.
\end{proof}


%
%

%
%
%

\subsection{The lower bound}
In this subsection we shall establish a lower bound of decay rate for the global solution $(a,u)$ associated with a class of initial data.
The decay rate of the lower bound is the same with that of the upper bound in Propositions \ref{Prop:upper.rough}, \ref{Prop:upper.refined} and thus it is indeed optimal.
\begin{proposition}\label{Prop:lower}
Let $\alpha\in(0,1)$ and $N\ge 2$.
Suppose that $(a_0,u_0) \in H^s \cap L^1(\R^N)$ with $s > \frac{N}{2}+1$.
If additionally $\int_{\R^N} a_0\dd x\neq 0$ and $\int_{\R^N} \rho_0 u_0\dd x\neq 0$,
there exists a constant $C_0>0$ such that
\begin{equation}\label{eq.decay.lower}
\begin{aligned}
  &\Vert a(t)\Vert_{L^2}\ge C_0\langle t\rangle^{-\frac{N}{2(2-\alpha)}},\quad \text{and }
  &\Vert u(t)\Vert_{L^2}\ge C_0\langle t\rangle^{-\frac{N+2(1-\alpha)}{2(2-\alpha)}}.
\end{aligned}
\end{equation}
\end{proposition}

In the previous subsections, we notice that the nonlinear terms decay faster than the linear ones.
For the large time, we thus expect that the influence of the linear terms dominates the nonlinear ones.
The following two lemmas provide the analysis of some linear terms.
\begin{lemma}\label{lem.lower.bound.G11G21}
Let $0<\alpha<1$. Suppose $\vert\xi\vert^{1-\alpha}<\mu/(2\sqrt\gamma)$, namely $\xi\in D$ where $D$ is defined in \eqref{eq:D}. Then we have
\begin{align}\label{eq.lower.bound.G11G21}
  &\vert \widehat{\mathcal{G}}_{11}(t,\xi) \vert \ge Ce^{-C\vert\xi\vert^{2-\alpha}t},\quad
  \vert \widehat{\mathcal{G}}_{21}(t,\xi) \vert \ge C\vert\xi\vert^{1-\alpha}
  \left(e^{-\frac{2\gamma}{\mu}t\vert\xi\vert^{2-\alpha}}-e^{-\frac{4\gamma}{\mu}t\vert\xi\vert^{2-\alpha}}\right).
\end{align}
\end{lemma}

\begin{proof}[Proof of Lemma \ref{lem.lower.bound.G11G21}]
From the definition of $\lambda_{\pm}$ in \eqref{def:lamb-pm}, we easily deduce that for every $\vert\xi\vert^{1-\alpha}<\mu/(2\sqrt\gamma)$,
\begin{align}\notag
  \lambda_{\pm}\le 0, \quad e^{t\lambda_{-}}>0\quad\text{and}
  \quad\frac{e^{t\lambda_{+}}-e^{t\lambda_{-}}}{\lambda_{+}-\lambda_{-}} \ge 0.
\end{align}
Thus by virtue of equality \eqref{eq.green11} and $\lambda_-\ge -2\frac{\gamma}{\mu}\vert\xi\vert^{2-\alpha}$, we have
\begin{align}\notag
  \vert\widehat{\mathcal{G}}_{11}(t,\xi)\vert
  =e^{t\lambda_{-}} -\frac{\lambda_{-}(e^{t\lambda_{+}}-e^{t\lambda_{-}})}{\lambda_{+}-\lambda_{-}} \ge e^{t\lambda_{-}}
  \ge e^{-\frac{2\gamma}{\mu}t\vert\xi\vert^{2-\alpha}}.
\end{align}
Due to that $\lambda_+\ge -\mu\vert\xi\vert^{\alpha}$ and $\mu\vert\xi\vert^{\alpha}>\frac{4\gamma}{\mu}\vert\xi\vert^{2-\alpha}$, we get
\begin{align}
  \notag\int_0^1e^{t(\theta\lambda_{+}+(1-\theta)\lambda_{-})}\dd \theta
  &\ge \int_0^{1}e^{-t(\theta{\mu}\vert\xi\vert^{\alpha}+2(1-\theta)\frac{\gamma}{\mu}\vert\xi\vert^{2-\alpha})}\dd \theta\\
  &\notag= t^{-1}({\mu}\vert\xi\vert^{\alpha}-\frac{2\gamma}{\mu}\vert\xi\vert^{2-\alpha})^{-1}
  \left(e^{-\frac{2\gamma}{\mu}t\vert\xi\vert^{2-\alpha}}-e^{-t{\mu}\vert\xi\vert^{\alpha}}\right)\\
  &\notag\ge \frac{1}{\mu}t^{-1}\vert\xi\vert^{-\alpha}
  \left(e^{-\frac{2\gamma}{\mu}t\vert\xi\vert^{2-\alpha}}-e^{-\frac{4\gamma}{\mu}t\vert\xi\vert^{2-\alpha}}\right).
\end{align}
Inserting the above estimate into \eqref{eq.green21} and  \eqref{eq:Green1}, we obtain the lower bound of $\vert \widehat{\mathcal{G}}_{21}\vert$ in \eqref{eq.lower.bound.G11G21}.
\end{proof}


\begin{lemma}\label{lem.lower.bound.G11G21.L2}
Let $0<\alpha<1$.
Suppose that $f\in L^2\cap L^1(\R^N)$ and $\widehat{f}(0) = \int_{\R^N} f(x) \dd x\neq 0$.
Then we have for every $t\ge 1$,
\begin{align}\label{eq:G-dec-lower}
  &\Vert \widehat{\mathcal{G}}_{11}(t)\widehat{f}\Vert_{L^2} \ge C\langle t\rangle^{-\frac{N}{2(2-\alpha)}},\quad
  \Vert \widehat{\mathcal{G}}_{21}(t)\widehat{f}\Vert_{L^2} \ge C\langle t\rangle^{-\frac{N+2(1-\alpha)}{2(2-\alpha)}}.
\end{align}
\end{lemma}

\begin{proof}[Proof of Lemma \ref{lem.lower.bound.G11G21.L2}]
The dominated convergence theorem and $f\in L^1(\R^N)$ imply $\widehat{f}\in C(\R^N)$.
Thus from $\widehat{f}(0)\neq 0$, there exist two positive constants $M>0$ and $0<\delta<{\mu^{\frac{1}{1-\alpha}}}/{(2\sqrt\gamma)^{\frac{1}{1-\alpha}}}$ such that $\vert\widehat{f}(\xi)\vert\ge M$
for every $\vert\xi\vert\le \delta$.
Owing to Lemma \ref{lem.lower.bound.G11G21}, we have that for every $t\ge 1$,
\begin{align}
  \Vert \widehat{\mathcal{G}}_{11}(t)\widehat{f}\Vert_{L^2}^2
  &\ge C\int_{D} e^{-2C\vert\xi\vert^{2-\alpha}t} \vert\widehat{f}\vert^2\dd \xi
  \ge CM^2\int_{\vert\xi\vert\le\delta} e^{-2C\vert\xi\vert^{2-\alpha}t} \dd \xi\notag\\
  &\ge CM^2t^{-\frac{N}{2-\alpha}}\int_{\vert\eta\vert\le\delta } e^{-2C\vert\eta\vert^{2-\alpha}} \dd \eta
  \ge C t^{-\frac{N}{2-\alpha}}.\notag
\end{align}
Similarly, for every $t\ge1$, we can similarly estimate $\widehat{\mathcal{G}}_{21}(t)\widehat{f}$ as follows:
\begin{align}
  \Vert \widehat{\mathcal{G}}_{21}(t) \widehat{f} \Vert_{L^2}^2
  &\ge CM^2\int_{\vert\xi\vert\le\delta} \vert\xi\vert^{2(1-\alpha)}
  \left(e^{-\frac{2\gamma}{\mu}t\vert\xi\vert^{2-\alpha}}-e^{-\frac{4\gamma}{\mu}t\vert\xi\vert^{2-\alpha}}\right)^2 \dd \xi\notag\\
  &\ge CM^2t^{-\frac{N+2(1-\alpha)}{2-\alpha}}\int_{\vert\eta\vert\le\delta }  \vert\eta\vert^{2(1-\alpha)}
  \left(e^{-\frac{2\gamma}{\mu}\vert\eta\vert^{2-\alpha}}-e^{-\frac{4\gamma}{\mu}\vert\eta\vert^{2-\alpha}}\right)^2 \dd \eta
  \ge C t^{-\frac{N+2(1-\alpha)}{2-\alpha}}.\notag
\end{align}
Since $t\le \langle t\rangle$, we thus complete the proof of \eqref{eq:G-dec-lower}.
\end{proof}

Based on Lemmas \ref{lem.lower.bound.G11G21} and \ref{lem.lower.bound.G11G21.L2},
we now give the proof of Proposition \ref{Prop:lower}.
\begin{proof}[Proof of Proposition \ref{Prop:lower}]
Taking advantage of \eqref{eq.CP}, \eqref{eq.G12.int.decay}, \eqref{eq.G11.F.decay}, \eqref{eq.G12.G.decay}
and Lemma \ref{lem.lower.bound.G11G21.L2}, we have that for $t\geq 1$,
\begin{align}\label{eq.sigma.L2.lower}
  \Vert a(t)\Vert_{L^2}
  &\ge \Vert\widehat{\mathcal{G}}_{11}(t)\widehat{a_0}\Vert_{L^2}
  -\Vert\widehat{\mathcal{G}}_{12}(t)\widehat{v_0}\Vert_{L^2}
  -\int_0^t\big(\Vert\widehat{\mathcal{G}}_{11}(t-\tau)\widehat{F}(\tau)\Vert_{L^2}
  +\Vert\widehat{\mathcal{G}}_{12}(t-\tau)\widehat{G}(\tau)\Vert_{L^2}\big)\dd \tau\notag\\
  &\ge 2C_1\langle t\rangle^{-\frac{N}{2(2-\alpha)}}-C\Big(\langle t\rangle^{-\frac{N+2(1-\alpha)}{2(2-\alpha)}}
  +\langle t\rangle^{-\frac{N+2}{2(2-\alpha)}}\Big).\notag
\end{align}
We thus find that for large enough $t$,
\begin{align}
\notag
  \Vert a(t)\Vert_{L^2}
  &\ge C_1\langle t\rangle^{-\frac{N}{2(2-\alpha)}}.
\end{align}
By virtue of \eqref{eq.CP}, \eqref{eq.G22.int.decay} \eqref{eq.G21.F.decay}, \eqref{eq.G22.G.decay}
and Lemma \ref{lem.lower.bound.G11G21.L2}, we get that for $t\geq 1$,
\begin{align}\notag
  \Vert v(t)\Vert_{L^2}
  &\ge \Vert\widehat{\mathcal{G}}_{21}(t)\widehat{a_0}\Vert_{L^2}
  -\Vert\widehat{\mathcal{G}}_{22}(t)\widehat{v_0}\Vert_{L^2}
  -\int_0^t \big( \Vert\widehat{\mathcal{G}}_{21}(t-\tau)\widehat{F}(\tau)\Vert_{L^2}
  +\Vert\widehat{\mathcal{G}}_{22}(t-\tau)\widehat{G}(\tau)\Vert_{L^2} \big) \dd \tau\\
  &\ge 2C_1\langle t\rangle^{-\frac{N+2(1-\alpha)}{2(2-\alpha)}}
  - C \langle t\rangle^{-\frac{N+4(1-\alpha)}{2(2-\alpha)}}.\notag
\end{align}
Hence for a sufficiently large $t$ we have
\begin{align}\notag
  \Vert v(t)\Vert_{L^2}
  &\ge C_1\langle t\rangle^{-\frac{N+2(1-\alpha)}{2(2-\alpha)}}.
\end{align}
Recall that $u=\mathbb{P}u-\nabla\Lambda^{-1} v$.
From $\Vert\nabla\Lambda^{-1} v\Vert_{L^2}=\Vert v\Vert_{L^2}$ and $\int_{\R^N} \mathbb{P}u\cdot\nabla\Lambda^{-1} v\dd x=0$, we have
\begin{align}\notag
  \Vert u(t)\Vert_{L^2}
  = \sqrt{\Vert \mathbb{P}u(t)\Vert_{L^2}^2+\Vert \nabla\Lambda^{-1} v(t)\Vert_{L^2}^2}
  \ge \Vert v(t)\Vert_{L^2}
  \ge C_1\langle t\rangle^{-\frac{N+2(1-\alpha)}{2(2-\alpha)}}.
\end{align}

In addition, from the system \eqref{eq.EAS} and \eqref{eq.EAS.a}, we have
$\int_{\R^N} a(t,x) \dd x = \int_{\R^N} a_0(x)\dd x\neq 0$
and the conservation of momentum $\int_{\R^N} \rho u(t,x)\dd x = \int_{\R^N} \rho_0 u_0(x)\dd x\neq 0$.
Therefore, there exists a constant $C_2>0$ such that $\Vert a(t)\Vert_{L^2}\ge C_2$ and $\Vert u(t)\Vert_{L^2}\ge C_2$.
Finally, these inequalities combined with the above estimates lead to \eqref{eq.decay.lower}, as desired.
\end{proof}

\begin{appendix}
\section{The expression of Green's matrix}
Here we give an exposition of the derivation of Green's matrix expression formula \eqref{eq:Green1}-\eqref{eq:Green2}.
The matrix $A$ is similar to its Jordan canonical form.
If $\lambda_+\neq\lambda_-$ (i.e. ${4\gamma}\vert \xi\vert^{2-2\alpha}\neq \mu^2$), we find that $A=SJS^{-1}$ with
\begin{equation*}J=
  \begin{pmatrix}
  \lambda_{+} & 0\\
  0 & \lambda_{-}
  \end{pmatrix},
\quad
S=
  \begin{pmatrix}
  \frac{-\lambda_{-}}{\gamma\vert \xi\vert} &\frac{-\lambda_{+}}{\gamma\vert \xi\vert}\\
  1 & 1
  \end{pmatrix},
  \quad
  S^{-1}
  = \frac{1}{\lambda_+ - \lambda_{-}}\begin{pmatrix}
  \gamma\vert \xi\vert  & \lambda_+\\
  -\gamma\vert \xi\vert  & -\lambda_{-}
  \end{pmatrix}.
\end{equation*}
If $\lambda_+=\lambda_-$ (i.e. ${4\gamma}\vert \xi\vert^{2-2\alpha}=\mu^2$), we have $A=SJS^{-1}$ with
\begin{equation*}
  J=
  \begin{pmatrix}
  -\sqrt\gamma\vert\xi\vert & 1\\
  0 & -\sqrt\gamma\vert\xi\vert
  \end{pmatrix},
  \quad
  S=
  \begin{pmatrix}
  \frac{1}{\sqrt\gamma} &\frac{1}{\gamma\vert\xi\vert}\\
  1 & 0
  \end{pmatrix},
  \quad
  S^{-1}= \begin{pmatrix}
  0 &1\\
  \gamma\vert\xi\vert &-\sqrt\gamma\vert\xi\vert
  \end{pmatrix}.
\end{equation*}
Hence, the fundamental solution matrix $\widehat{\mathcal{G}}(t)=e^{At}=Se^{Jt}S^{-1}$
has the wanted formula \eqref{eq:Green1}-\eqref{eq:Green2}.

\end{appendix}

\bibliographystyle{plain}

\end{document}